\documentclass[oneside,12pt]{amsart}
\usepackage{amsmath}
\usepackage{amsthm}
\usepackage{amsfonts}
\usepackage{amssymb}
\usepackage{latexsym}
\usepackage{faktor}
\usepackage{color}
\usepackage{tikz-cd}
\usepackage{url}
\usepackage{doc}

\title{Selective separability properties of Fr\'echet-Urysohn spaces and their products}

\author[S. Bardyla]{Serhii Bardyla}
\address{Universit\"at Wien, Institut f\"ur Mathematik, Kurt G\"odel Research Center, Kolingasse 14-16, 1090 Vienna, Austria}
\email{sbardyla@gmail.com}
\urladdr{http://www.logic.univie.ac.at/~bardylas55/}

\author[F. Maesano]{Fortunato Maesano}
\address{MIFT - Matematica e Informatica, Scienze Fisiche e Scienze della Terra Department, University of Messina, Viale F. Stagno d'Alcontres 31, 98166 Messina, Italy.}
\email{fomaesano@unime.it}

\author[L. Zdomskyy]{Lyubomyr Zdomskyy}
\address{Institut f\"ur Diskrete Mathematik und Geometrie, Technische Universit\"at Wien, Wiedner
Hauptstrasse 8-10/104, 1040 Wien, Austria.}
\email{lzdomsky@gmail.com}
\urladdr{https://dmg.tuwien.ac.at/zdomskyy/}

\thanks{The first author was supported by the Austrian Science Fund FWF (Grant ESP 399).
The second author would like to thank the ``National Group for Algebraic and Geometric Structures, and their Applications'' (GNSAGA – INdAM) for generous support for this research. The third author would like to thank  the Austrian Science Fund FWF (Grants I 3709 and I 5930) for generous support for this research.}

\date{}

\subjclass[2020]{Primary:  	03E50, 54D65, 03E35; Secondary: 54D10,
 	03E17, 03E65.}

\keywords{Frech\'{e}t-Urysohn space, Martin's Axiom, separated mad family, $M$-separability, $H$-separability, $R$-separability,  $ \alpha_{i} $-space.}

\newtheorem{theorem}{Theorem}[section]

\newtheorem{question}[theorem]{Question}

\newtheorem{lemma}[theorem]{Lemma}
\newtheorem{proposition}[theorem]{Proposition}

\newtheorem{definition}[theorem]{Definition}

\newtheorem{claim}{Claim}[theorem]

\makeatletter
\def\@makechapterhead#1{%
  \vspace*{50\p@}%
  {\parindent \z@ \raggedright \normalfont
    \ifnum \c@secnumdepth >\m@ne
      \if@mainmatter
        \Huge\bfseries \thechapter.\space%
      \fi
    \fi
    \interlinepenalty\@M
    \Huge \bfseries #1\par\nobreak
    \vskip 40\p@
  }}
\makeatother

\newcommand{\bb}{\mathfrak b}
\newcommand\extt[2]{\ensuremath{#1^{\,\smallfrown}\mkern-2mu #2}}
\newcommand{\forces}{\Vdash}
\newcommand{\la}{\langle}
\newcommand{\name}[1]{\dot{#1}}
\newcommand{\ra}{\rangle}
\newcommand{\hot}{\mathfrak}
\newcommand{\zrost}{\w^{\uparrow\w}}

\newcommand{\w}{\omega}
\newcommand{\A}{\mathcal{A}}
\newcommand{\C}{\mathcal{C}}
\newcommand{\F}{\mathcal{F}}
\newcommand{\G}{\mathcal{G}}
\newcommand{\I}{\mathcal{I}}
\newcommand{\IP}{\mathbb P}
\newcommand{\IQ}{\mathbb Q}
\newcommand{\PP}{\mathcal P}

\newcommand{\U}{\mathcal U}
\newcommand{\cU}{\mathcal U}
\newcommand{\Y}{\mathcal Y}
\newcommand{\Z}{\mathcal Z}

\newcommand{\uhr}{\upharpoonright}

\begin{document}

\begin{abstract}
In this paper we study the behaviour of selective separability properties in the class of Frech\'{e}t-Urysohn spaces. We present two examples, the first one given in ZFC proves the existence of a countable Frech\'{e}t-Urysohn (hence $R$-separable and selectively separable) space which is not $H$-separable; assuming $\mathfrak{p}=\mathfrak{c}$, we construct such an example which is also zero-dimensional and $\alpha_{4}$. Also,  motivated by a result of Barman and Dow stating that
the product of two countable Frech\'{e}t-Urysohn  spaces is $M$-separable under PFA,  we show that the MA is not sufficient here.
In the last section we prove that in the Laver model, the product of any two $H$-separable spaces is $mH$-separable.
\end{abstract}

\maketitle

\section{Introduction}

A space $X$ is said to be \emph{Frech\'{e}t-Urysohn}, briefly FU, if for every $A \subset X$ and $x \in \overline{A}\setminus A$ there exists a  sequence $S \subset A\setminus\{x\}$ converging to $x$, i.e., $S\in [A\setminus\{x\}]^\w$ and $|U\setminus A|<\w$ for every open neighbourhood $U$ of $x$.
It has been known since 50 years ago that the product of two FU spaces
does not have to be a FU space, neither for  countable nor for compact spaces,
see \cite{Costantini99, Nogura85, Ols74, Simon08} and references therein.
However, it is still interesting which properties
products of two FU spaces must (at least consistently) possess.
In this paper we study various combinatorial versions of the separability of such products. Our motivation comes from the following result of
Barman and Dow:
\begin{theorem}\cite[3.3]{BarDow12} \label{pfa_th}
	(PFA) The product of finitely many countable FU spaces is M-separable.
\end{theorem}
Recall from \cite{Scheep99} that a topological space $X$ is said to  be  \emph{selectively separable} (or \emph{M-separable}) if for every sequence $\la D_{n} : n \in \omega \ra$ of dense subsets of $X$, there are finite sets $F_{n} \subset D_{n}$, $n \in \omega$, such that $\bigcup \{F_{n}: n \in \omega\}$ is dense in $X$.

In contrast with Theorem~\ref{pfa_th}, Barman and Dow  proved  in \cite{BarDow11} that every separable FU space is $M$-separable and that under CH, there are two
countable FU spaces whose product is not $M$-separable.
In \cite[6.1 and 6.2]{MilTsaZdo16}
this result has been improved by showing that under CH there are two countable FU $H$-separable topological groups whose
product is not $M$-separable. The main result of this paper
is the following
\begin{theorem} \label{th:main}
It is consistent with (MA+$\neg\mathit{CH}$) that there are
two countable regular FU spaces whose product is not $M$-separable.
\end{theorem}
Among other things Theorem~\ref{th:main} shows that Theorem~\ref{pfa_th}
cannot be proved using only the equality between  some standard
cardinal characteristics of the continuum, as all these characteristics are also equal to $\mathfrak c$ under MA.
The proof of Theorem~\ref{th:main} is based on the application
of certain kind of separated mad families, whose existence
is consistent with MA by essentially the same argument as in
\cite{DowShe12}, see Section~\ref{sec4}.

The notion of $M$-separability motivated
Bella, Bonanzinga and Matveev to introduce in \cite{BELLA20091241}
further  variations thereof of this property. In particular,  a topological space $X$ is said to  be \emph{H-separable} (resp., \emph{R-separable}), if for every sequence $\la D_{n} : n \in \omega \ra$ of dense subsets of $X$, there are finite sets $F_{n} \subset D_{n}$  (resp., points $x_{n} \in D_{n}$), $n \in \omega$, such that every nonempty open set in $X$ meets all but finitely many $F_{n}$ (resp.,  $\{x_{n}: n \in \omega\}$ is dense in $X$).

Gruenhage and Sakai in \cite{GRUENHAGE20111352} proved that separable FU  spaces are $R$-separable (this  improves the result of Barman and Dow asserting that separable FU spaces are $M$-separable). In Section~\ref{sec2} we prove that the $H$-separability behaves  differently, i.e., in ZFC  there exists a  countable Hausdorff FU space which is not $H$-separable, see Theorem~\ref{ma_weaker_pfa}. This gives new examples of a space which is $R$-separable but not $H$-separable, see \cite{BELLA20091241} and \cite{Bella2008SelectiveSG} for more examples of this kind.

In \cite{GRUENHAGE20111352}, using different terminology, the following weaker versions of combinatorial density properties were introduced:  A topological space $X$ is said to be
\emph{$mP$-separable,} where $P\in\{M,H,R\}$, if
it satisfies the conclusion of the $P$-separability for any decreasing sequence of dense subsets of $X$. For instance,
$X$ is \emph{mR-separable}  if for every decreasing sequence $\la D_{n} : n \in \omega \ra$ of dense subsets of $X$, there are points $x_{n} \in D_{n}$ such that $\{x_{n}: n \in \omega\}$ is dense in $X$.
The following diagram sums up the implications between the cited selective properties.
\[
\begin{tikzcd}
	R\text{-}sep. \arrow[r] \arrow[d]
	&M\text{-}sep. \arrow[d]
	&\arrow[l]  H\text{-}sep.  \arrow[d]\\
	mR\text{-}sep. \arrow[r]
	&\arrow[u] mM\text{-}sep.
	&\arrow[l]  mH\text{-}sep.
\end{tikzcd}
\]
It is proved in \cite{GRUENHAGE20111352} that the $M$-separability and $mM$-separability are in fact equivalent, which motivated the authors to ask in \cite[Question 2.10(1)]{GRUENHAGE20111352} \emph{whether
there exists an $mH$-separable space which is not $H$-separable.}
Since \cite[Lemma 2.7(2)]{GRUENHAGE20111352} combined with \cite[Corollary 4.2]{GRUENHAGE20111352} implies that every countable FU space is $mH$-separable, the FU non-$H$-separable space  that we build in the proof of Theorem~\ref{ma_weaker_pfa} is  $mH$-separable,
and thus our Theorem~\ref{ma_weaker_pfa} answers the aforementioned question posed in \cite{GRUENHAGE20111352} in the affirmative.
We do not know how to get a regular (equivalently, zero-dimensional)
example outright in ZFC.

\begin{question}
\begin{enumerate}
\item Is there a ZFC example of a regular countable FU (resp. $mH$-separable) space which is not $H$-separable?
\item Is there a (ZFC example of) a [regular] countable $mR$-separable space
which is not $R$-separable?
\end{enumerate}
\end{question}

Given a space $X$ and $x \in X$, we denote by $\Gamma_{x}$  the set of all $A \in [X\setminus\{x\}]^{\omega}$ which converge to $x$. In \cite{Arh72}, Arhangel'skiĭ introduced the following local  properties
of a space $X$ at  some point $x\in X$:
\begin{itemize}
	\item[$(\alpha_{1})$] For each    $\la S_{n} : n \in \omega \ra \in \Gamma^\w_{x}$, there is  $S \in \Gamma_{x}$ such that $S_{n}\subset^* S$  for all $n \in \omega$;
	\item[$(\alpha_{2})$]  For each    $\la S_{n} : n \in \omega \ra \in \Gamma^\w_{x}$, there is  $S \in \Gamma_{x}$ such that $S_{n} \cap S$ is infinite for all $n \in \omega$;
	\item[$(\alpha_{3})$]  For each    $\la S_{n} : n \in \omega \ra \in \Gamma^\w_{x}$, there is  $S \in \Gamma_{x}$ such that $S_{n} \cap S$ is infinite for infinitely many $n \in \omega$;
\item[$(\alpha_{4})$]  For each    $\la S_{n} : n \in \omega \ra \in \Gamma^\w_{x}$, there is  $S \in \Gamma_{x}$ such that $S_{n} \cap S\neq\emptyset$ for infinitely many $n \in \omega$.
\end{itemize}
A space $X$ is an $\alpha_i$ space, where $i\in\{1,2,3,4\}$,
if it is an $\alpha_i$ space at each $x\in X.$
Each of these $\alpha_{i}$-properties obviously  imply the next one.
The idea behind the introduction of $\alpha_i$ spaces was  to find additional properties which together with the FU property guarantee that
the product of two spaces is again FU, see \cite{Arh72,Nogura85}
for more details.

In Section~\ref{sec3} we  prove that every separable $\alpha_{2}$  FU space is $H$-separable,  and $\mathfrak{p}=\mathfrak{c}$ yields a countable regular FU $\alpha_{4}$ which is not  $H$-separable.
This leaves the following question open, which actually consists of several
subquestions.
\begin{question}
Is it consistent that every (regular) $\alpha_4$ [resp. $\alpha_3$]
FU space is $H$-separable?
\end{question}

In Section~\ref{sec5}  we show that
it is consistent that the product of two countable $H$-separable spaces
is $mH$-separable. More precisely, this is the case in the classical Laver model introduced in \cite{Lav76}.
This improved an earlier result from   \cite{RepZd18}
stating that in the same model,  the product of any two countable $H$-separable spaces is $M$-separable.
Consequently, in the Laver model the product of any two $H$-separable spaces is $mH$-separable provided that it is hereditarily separable.
This result motivates the following
\begin{question}
\begin{enumerate}
\item Is there a countable (regular) $mH$-separable space in the Laver model
which is not $H$-separable?
\item Is the class of $H$-separable ($mH$-separable) spaces
consistently closed under finite products? What about the Laver model?
\item Is the product of 3 (resp. finitely many)
$H$-separable spaces $mH$-separable in the Laver model?
\end{enumerate}
\end{question}

Finally, the proof of Theorem~\ref{pfa_th} gives actually the $mR$-separability of the product, but we do not know whether it can be
pushed further.
\begin{question}
Does PFA imply that the product of two countable
FU spaces is $mH$- (resp. $R$-) separable?
\end{question}

All the undefined notions can be found in
\cite{Bla10},
\cite{Engelking}  and \cite{Kunen}.

\section{A Hausdorff Fr\'echet-Urysohn space which is not $H$-separable in ZFC} \label{sec2}

In the proof of the main result of this section we
 shall use a fundamental result of Mathias stating that there
is no analytic  mad family $\A\subset[\w]^\w$, see \cite[Corollary 4.7]{Mat77}. Recall that an infinite $\A\subset[\w]^\w$ is \emph{mad},
if any two different elements of $\A$ have finite intersection,
and $\A$ is not included into to a strictly bigger family with this property.

\begin{theorem} \label{zfc}
There exists a countable Hausdorff Frechet-Urysohn
 space $X$
 without isolated points which is not $H$-separable.
\end{theorem}
\begin{proof}
 The underlying set of $X$ will be $\w$, and let
$\tau_0$ be a topology on $\w$ such that $\la \w,\tau_0\ra$ is
homeomorphic to the rationals.

The final topology $\tau=\tau_{\hot c}$ turning $\w$ into a space
with the needed properties  will be constructed recursively over
ordinals $\alpha\in\hot c$ as an increasing union $\tau_{\hot
c}=\bigcup_{\alpha<\hot c}\tau_\alpha$ of intermediate topologies.


 Let $\la E_n:n\in\w\ra$ be a sequence of
mutually disjoint dense subsets of $\la \w,\tau_0\ra$.
 We shall make sure that $E_n$'s remain dense in
 $\la \w,\tau_\alpha\ra$ for all $\alpha\leq\hot c$.
Let $\{\la S_\alpha,x_{\alpha}\ra:\alpha<\hot c\}$
 be an
enumeration of $[\w]^\w\times\w$ such that for each $\la S,x\ra\in
[\w]^\w\times\w$ there are cofinally many ordinals $\alpha$ such
that $\la S,x\ra=\la S_\alpha,x_\alpha\ra$. Fix an enumeration
$\{\la F^\alpha_n:n\in\w\ra:\alpha<\hot c\}$ of
$\prod_{n\in\w}[E_n]^{<\w}$. We shall also need a mad family  on
$\w$ of size $\hot c$ obtained as follows. Let $\A\subset [\w]^\w$
be a compact almost disjoint family of size $\hot c$. It can easily
be constructed by, e.g., considering the family of all branches
through $2^{<\w}$, and then copying it via any bijection between
$\w$ and $2^{<\w}$. Let $\C$ be any mad family extending $\A$. Note
that $\C\neq\A$ since there are no analytic mad families by \cite[Corollary 4.7]{Mat77}.

Suppose that for some $\alpha<\hot c$ and all $\delta\in\alpha$ we
have already constructed a topology $\tau_\delta$ on $\w$, a family
$\Y_\delta\subset [\w]^\w$, and for every $Y\in \Y_\delta$ either
$n(Y)\in\w$ (such $Y$ will be called \emph{vertical}) or $C(Y)\in\C$
(such $Y$ will be called \emph{horizontal})
 such that the following conditions are satisfied:
\begin{itemize}
\item[$(1)$] $\tau_\beta\subset \tau_\delta$ for all
$\beta\leq\delta$;
\item[$(2)$]
 $\Y_\beta\subset \Y_\delta$ for all $\beta\leq\delta$
and $\Y_\delta $ consists of sequences convergent in
$\la\w,\tau_\delta\ra$;
\item[$(3)$]
For every $\beta<\delta$, if  $x_\beta$ is a limit point  of
$S_\beta$, then
  there exists $Y\in\Y_\delta $ converging to $x_\beta$ with
$|Y\cap S_\beta|=\w$;
\item[$(4)$]
Every vertical $Y\in\Y_\delta$ is contained in $E_{n(Y)}$, and each
horizontal $Y\in\Y_\delta$ is a subset of
  $\bigcup_{n\in C(Y)}E_n$;
Moreover, in the latter case  $|Y\cap E_n|\leq 1$ for all $n\in
C(Y)$;
\item[$(5)$] For every $\beta<\delta$ there exists $U_\beta\in\tau_\delta$
and $A_\beta\in \A$ such that $ U_\beta=\w\setminus\bigcup_{n\in
A_\beta}F^\beta_n,$ and
$\{A_\beta:\beta<\delta\}\cap\{C(Y):Y\in\Y_\delta\}=\emptyset$;
\item[$(6)$] $\tau_\delta$ is generated by $\tau_0\cup\{U_\beta:\beta<\delta\}$
as a subbase;
\item[$(7)$] $E_n$ is dense in $\la\w,\tau_\delta\ra$ for all
$n\in\w$.
\end{itemize}
Let us note that $(5)$ and $(6)$ imply $(7)$, as well as the fact
that for every  $\C'\in [\C]^{<\w}$ and $K\in [\w]^{<\w}$ we have
that $\tau_\delta\uhr\bigcup_{n\in K\cup \bigcup\C'}E_n$ has a
countable base. Moreover, if $\C'\in
[\C\setminus\{A_\beta:\beta<\delta\}]^{<\w}$, then
$$\tau_\delta\uhr\bigcup_{n\in K\cup \bigcup\C'}E_n=
\tau_0\uhr\bigcup_{n\in K\cup \bigcup\C'}E_n.$$

We shall consider several cases.
\smallskip

\textbf{\textit{I}}. $\alpha$ is limit. \\
It is easily checked that the topology $\tau_\alpha$ generated by
$\bigcup_{\delta<\alpha}\tau_\delta$ as a base,  along with
$\Y_\alpha=\bigcup_{\delta<\alpha}\Y_\delta$ satisfies
 $(1)$-$(7)$
for  $\delta=\alpha$.
\smallskip

\textbf{\textit{II}}.  $\alpha=\delta+1$ and    $x_\delta$ is a limit point
of $S_\delta$ in $\la\w,\tau_\delta\ra$.
\\
 If $|Y\cap
S_\delta|=\w$ for some $Y\in\Y_\delta$ converging to $x_\delta$,
then we set $\Y_\alpha=\Y_\delta$,  pick any
$$A_\delta\in\A\setminus\{C(Y):Y\in\Y_\delta\},$$
 let $U_\delta$ be such as in item $(6)$,
denote by $\tau_\alpha$ the topology generated by
$\tau_\delta\cup\{U_\delta\}$ as a subbase,
 and note that all the
conditions $(1)$-$(7)$ are satisfied. E.g., each
$Y\in\Y_\alpha=\Y_\delta$ is convergent also in
$\la\w,\tau_\alpha\ra$ since $U_\delta$ almost contains all such
$Y$.

 So consider
the case $|Y\cap S_\delta|<\w$ for all $Y\in\Y_\delta$ converging to
$x_\delta$. If there exists $n\in\w$ such that
$x_\delta\in\overline{S_\delta\cap E_n}^{\tau_\delta}$, let
$Y_\delta\in [S_\delta\cap E_n]^\w$ be a sequence converging to
$x_\delta$ (remember that all the $\tau_\beta$'s restricted to $E_n$
are the same, and thus turn $E_n$ into a copy of the rationals), and
in this case we set $n(Y_\delta)=n$. Let us assume now that  there
is no such $n$. We claim that there exists a sequence $Y\in
[S_\delta]^\w$ converging to $x_\delta$ in $\la\w,\tau_\delta\ra$.
Let $\C'$  be the family of all those $C\in\C$ such that there
exists a sequence $Y_C\in [S_\delta]^\w$ convergent to $x_\delta$ in
$\la\w,\tau_0\ra$ such that $Y_C\subset\bigcup_{n\in C}E_n$ and
$|Y_C\cap E_n|\leq 1$ for all $n\in C$. Two cases are possible.

a)\   $\C'$ is finite.\\
It follows that $x_\delta$ is not in the closure of
$S_\delta'=S_\delta\setminus\bigcup\{E_n:n\in\cup\C'\}$ with respect
to $\tau_0$. Indeed, otherwise there exists a sequence $Y\in
[S_\delta']^\w$ converging to $x_\delta$ in $\la\w,\tau_0\ra$, and
since it cannot have infinite intersection with any $E_n$, we
conclude that there exists $L\in [\w]^\w$ with $|Y\cap
E_n|\neq\emptyset$ for all $n\in L$. Then $L\cap\cup\C'=\emptyset$,
but on the other hand
 any $C\in\C$ with $|C\cap L|=\w$ must be in $C'$, a contradiction.
Consequently, $x_\delta$ is also not in the closure of $S_\delta'$
with respect to $\tau_\delta$, and hence it must be in the closure
of $S_\delta''=S_\delta\cap\bigcup\{E_n:n\in\cup\A'\}$ with respect
to $\tau_\delta$. However, $\tau_\delta\uhr S_\delta''$ has
countable base, so there exists a sequence $Y_\delta\in
[S_\delta'']^\w$ convergent to $x_\delta $ with respect to
$\tau_\delta$. By shrinking $Y_\delta$, if necessary, we may assume
find $C(Y_\delta)\in\C'$ satisfying $(4)$ along with $Y_\delta$. It
remains to set $\Y_\alpha=\Y_\delta\cup\{Y_\delta\}$,
 pick any $A_\delta\in\A\setminus\{C(Y):Y\in\Y_\alpha\}$,
 and let $U_\delta$ and $\tau_\alpha$ be such as required in
$(5)$, $(6)$ for $\alpha$.

b)\  $\C'\supset\{C_k:k\in\w\}$, where $C_k\neq C_m$ for $k\neq m$.\\
For every $k\in\w$ let $Y_{C_k}$ be witnessing for $C_k\in\C'$. By
shrinking $Y_{C_k}$'s, if necessary, we may assume that
$|\bigcup_{k\in\w}Y_{C_k}\cap E_n|\leq 1$ for all $n\in\w$. Let
$Y\subset\bigcup_{k\in\w}Y_{C_k}$ be such that $Y_{C_k}\subset^*Y$
for all $k\in\w$ and $Y$ converges to $x_\delta$ with respect to
$\tau_0$. Such  $Y$ obviously exists since $\tau_0$ has a countable
base. Set $L=\{n\in\w:Y\cap E_n\neq\emptyset\}$ and note that $L$
cannot be covered by finitely many elements of $\C$ because $|L\cap
C_k|=\w$ for all $k\in\w$. Thus
$$\{A\cap L\: :\: A\in\A, |A\cap L|=\w\}$$
cannot be a mad family of infinite subsets of $L$ since it is
analytic, and hence there exists $L'\in [L]^\w$  which is almost
disjoint from all $A\in\A$. Let $C\in\C$ be such that $|C\cap
L'|=\w$, denote by $Y_\delta$ the sequence $Y\cap\bigcup_{n\in C\cap
L'}E_n$, set $C(Y_\delta)=C$, and
$Y_\alpha=\Y_\delta\cup\{Y_\delta\}$. Finally, pick any $A_\delta\in
A\setminus\{C(Y):Y\in\Y_\alpha\}$  and let $U_\delta$ and
$\tau_\alpha$ be such as required in $(5)$, $(6)$ for $\alpha$.
Since $A_\beta\cap C(Y)$ is finite for all $Y\in\Y_\alpha$ and
$\beta<\alpha$, we conclude that $Y\subset^* U_\beta$ for all
$Y\in\Y_\alpha$ and $\beta<\alpha$, and hence all $Y\in\Y_\alpha$
remain convergent with respect to $\tau_\alpha$.
\smallskip

\textbf{\textit{III}}.  $\alpha=\delta+1$ and    $x_\delta$ is not a limit
point of $S_\delta$ in $\la\w,\tau_\delta\ra$.
\\
In this case we set $x_\delta'=0$ (any $i\in\w$ instead of $0$ would
work), $S_\delta'=\w\setminus\{0\}$, and repeat what we have done in
item $II$, with $\la S_\delta',x_\delta'\ra$ instead of $\la
S_\delta,x_\delta\ra$. Note that $x_\delta'$ must be in the closure
of $S_{\delta}'$ in $\la\w,\tau_\delta\ra$ since this topological
space has no isolated points, because it has disjoint dense subsets
by $(7)$.

This completes our recursive construction of the objects mentioned
in $(1)$-$(7)$ so that these conditions are satisfied. The space
$\la\w,\tau_{\hot c}\ra$ is as required: $(2),(3)$ imply the
Frechet-Urysohn property, and $(5)$ entails the failure of the
$H$-separability.
\end{proof}


\section{A Fr\'echet-Urysohn $\alpha_{4}$ space which is not $H$-separable under $\hot p=\hot c$} \label{sec3}

The main result of this section may be thought of as a  version of Theorem~\ref{zfc}, namely we shall get a zero-dimensional $\alpha_4$ example like in Theorem~\ref{zfc}, but at the cost of an additional set-theoretic assumption.

The next easy statement can be considered as a folklore and we shall
refer to it as ``there are no $(\w,<\hot b)$-gaps''. A very similar
argument can  be found in \cite[page 578]{Jec03}.
\begin{lemma}\label{no_w_<b_gaps}
Suppose that $\A\subset [\w]^\w$, $|\A|<\hot b$, $\{B_i:i\in\w\}\subset [\w]^\w$,
and $|A\cap B_i|<\w$ for all $A\in\A$ and $i\in\w$. Then there exists
$X\subset\w$ such that $A\subset^*\w\setminus X$ and $B_i\subset^* X$ for any $A\in\A$
and $i\in\w$.
\end{lemma}
\begin{proof}
For each $A\in\A$ find an increasing  function $f_A\in\w^\w$ with $A\cap B_i\subset f_A(i)$ for all $i\in\w$,
and let $f\in\w^\w$ be such that $f_A\leq^* f$ for all $A\in\A$. Then $X=\bigcup_{i\in\w}(B_i\setminus f(i))$
is as required.
\end{proof}

\begin{theorem}\label{p=c}
($\hot p=\hot c$) There exists a countable zero-dimensional
$\alpha_4$ Frechet-Urysohn
 space $X$
 without isolated points which is not $H$-separable.
\end{theorem}
\begin{proof}
Let $\G$ be a $P_{\hot c}$-point, i.e., an ultrafilter such that for
any $\G'\in [\G]^{<\hot c}$ there exists $G\in\G$ such that
$G\subset^* G'$ for all $G'\in\G'$. It is well-known and easy to
check that $\hot p=\hot c $ implies the existence of such an
ultrafilter. The underlying set of $X$ will be $\w$, and let
$\tau_0$ be a topology on $\w$ such that $\la \w,\tau_0\ra$ is
homeomorphic to the rationals.

The final topology $\tau=\tau_{\hot c}$ turning $\w$ into a space
with the needed properties  will be constructed recursively over
ordinals $\alpha\in\hot c$ as an increasing union $\tau_{\hot
c}=\bigcup_{\alpha<\hot c}\tau_\alpha$ of intermediate
zero-dimensional topologies.


 Let $\la E_n:n\in\w\ra$ be a sequence of
mutually disjoint dense subsets of $\la \w,\tau_0\ra$.
 We shall make sure that $E_n$'s remain dense in
 $\la \w,\tau_\alpha\ra$ for all $\alpha\leq\hot c$.
Let $\{\la S_\alpha,\A_\alpha,x_{\alpha}\ra:\alpha<\hot c\}$
 be an
enumeration of $[\w]^\w\times [[\w]^\w]^\w\times\w$ such that for
each $\la S,\A,x\ra\in [\w]^\w\times [[\w]^\w]^\w\times\w$ there are
cofinally many ordinals $\alpha$ such that $\la S,\A,x\ra=\la
S_\alpha,\A_\alpha,x_\alpha\ra$. Fix also an enumeration $\{\la
F^\alpha_n:n\in\w\ra:\alpha<\hot c\}$ of
$\prod_{n\in\w}[E_n]^{<\w}$.

Suppose that for some $\alpha<\hot c$ and all  $\delta\in\alpha$ we
have already constructed a zero-dimensional topology $\tau_\delta$
on $\w$, an almost disjoint family $\Y_\delta\subset [\w]^\w$, and
for every $Y\in Y_\delta$ and element $I_Y\in\G^*=\mathcal
P(\w)\setminus\G$
 such that the following conditions are satisfied:
\begin{itemize}
\item[$(1)$] The weight of $\tau_\delta$ is $<\hot c$;
\item[$(2)$] For every $n,x\in\w$ there exists $Y\in\Y_0$
such that $Y\in [E_n]^\w$ and $Y$ converges to $x$;
\item[$(3)$]
 $\Y_\beta\subset \Y_\delta$ for all $\beta\leq\delta<\alpha$
and $\Y_\delta $ consists of sequences convergent in $\la\w,\tau_\delta\ra$;
\item[$(4)$]
For every $\beta<\delta$, if  $x_\beta$ is a limit point  of
$S_\beta$ and every $A\in\A_\beta$ is a sequence converging to
$x_\beta$
 in $\la\w,\tau_\delta\ra$, then
 \begin{itemize}
\item there exists $Y\in\Y_\delta $ converging to $x_\beta$ with
$|Y\cap S_\beta|=\w$;
\item there
exists $Y\in\Y_\delta$ such that $Y\cap
(A\setminus\{x_\beta\})\neq\emptyset$ for infinitely many
$A\in\A_\beta$;
 \end{itemize}
\item[$(5)$]
Every $Y\in\Y_\delta$ is contained in $\bigcup_{n\in I_Y}E_n$;
Moreover, either $|Y\cap E_n|=1$ for all $n\in I_Y$, or $|I_Y|=1$;
\item[$(6)$] For every $\beta<\delta$ there exists $U_\beta\in\tau_\delta$
and $G_\beta\in\G$ such that $ U_\beta\cap\bigcup_{n\in
G_\beta}F^\beta_n=\emptyset; $
\item[$(7)$] $E_n$ is dense in $\la\w,\tau_\delta\ra$ for all
$n\in\w$.
\end{itemize}
We shall consider several cases.
\smallskip

\textbf{\textit{I}}. $\alpha$ is limit. \\
It is easily checked that the topology $\tau_\alpha$ generated by
$\bigcup_{\delta<\alpha}\tau_\delta$ as a base,  along with
$\Y_\alpha=\bigcup_{\delta<\alpha}\Y_\delta$ satisfies
 $(1)$-$(7)$
for  $\delta=\alpha$.
\smallskip

\textbf{\textit{II}}.  $\alpha=\delta+1$,    $x_\delta$ is a limit point of
$S_\delta$ in $\la\w,\tau_\delta\ra$, and $\A_\delta$ consists of
mutually disjoint sequences convergent to $x_\delta$.\\
 If $|Y\cap
S_\delta|=\w$ for some $Y\in\Y_\delta$ converging to $x_\delta$,
then we denote by $Y_{0,\delta}$ one of these $Y$'s. Similarly, if
$\big\{A\in\A_\beta\: :\: Y\cap (A\setminus\{x_\beta\})\neq
\emptyset\big\}$ is infinite for some $Y\in\Y_\delta$, we denote by
$Y_{1,\delta}$ one of these $Y$'s.

So consider the case  $|Y\cap S_\delta|<\w$ for all $Y\in\Y_\delta$
converging to $x_\delta$.
 Since the weight of
$\la\w,\tau_\delta\ra$ is $<\hot c=\hot p$, there exists a sequence
$Y_{0,\delta}\in [S_\delta]^\w$ convergent to $x_\delta$. Passing to
an infinite subset of $Y_{0,\delta}$, if necessary, we may assume
that either there exists $n\in\w$ such that $Y_{0,\delta}\subset
E_n$, in which case we set $I_{Y_{0,\delta}}=\{n\}$, or there exists
an infinite $I_{Y_{0,\delta}}\in\G^*$ such that
$Y_{0,\delta}\subset\bigcup_{n\in I_{Y_{0,\delta}}}E_n$ and
$|Y_{0,\delta}\cap E_n|=1$ for all $n\in I_{Y_{0,\delta}}$. Note
that  $Y_{0,\delta}$ is almost disjoint from any $Y\in\Y_{\delta}$.

Now suppose that  $\{A\in\A_\beta\: :\: Y\cap
(A\setminus\{x_\beta\})\neq \emptyset\}$ is finite for all $Y\in\A$.
Replacing each $A\in\A_\delta$ with an infinite subset thereof, we
may assume that for each $A\in\A_\delta$, either there exists
$n(A)\in\w$ such that $A\subset E_{n(A)}$ (such $A$ will be called
\emph{vertical}), or $|A\cap E_n|\leq 1$ for all $n\in\w$ (such $A$
will be called \emph{horizontal}). Since every space of character
$<\hot b$ is $\alpha_1$, there exists a sequence
$Y\subset\cup\A_\delta$ convergent to $x_\delta$ such that $|Y\cap
A|=1$ for all $A\in\A_\delta$. The fact that no element of
$\Y_\delta$ converging to $x_\delta$ intersects infinitely many
$A\in\A_\delta$ yields that $Y$ is almost disjoint from all elements
of $\Y_\delta$.
 Note also that $\A_\delta$ is a
disjoint family, and hence each infinite subset of $Y$ intersects
infinitely many elements of $\A_\delta$.

 If there are infinitely many
horizontal sequences, then $|Y\cap E_n|\neq\emptyset$ for infinitely
many $n$, and hence by shrinking $Y$ to some infinite $Y_{1,\delta}$
we may assume that $|Y_{1,\delta}\cap E_n|\leq 1$ for all $n$ and
$|Y_{1,\delta}\cap E_n|= 1$ if and only if  $n\in I$ for some
$I\in\G^*$. In this case we set $I_{Y_{1,\delta}}=I$.

Now suppose that all but finitely many $A\in\A_\delta$ are vertical.
If $\{n(A):A\in\A_\delta,A$ is vertical$\}$ is infinite, then as
before by shrinking $Y$ to some infinite $Y_{1,\delta}$ we may
assume that $|Y_{1,\delta}\cap E_n|\leq 1$ for all $n$ and
$|Y_{1,\delta}\cap E_n|= 1$ if and only if $n\in I$ for some
$I\in\G^*$. Also in this case we set $I_{Y_{1,\delta}}=I$.

In the remaining case there exists $n\in\w$ with
$\{A\in\A_\delta:n(A)=n\}$ infinite. Then  we set
$Y_{1,\delta}=Y\cap E_n$, $I_{Y_{1,\delta}}=\{n\}$, and note that
$Y_{1,\delta}$ is infinite.

Shrinking $Y_{0,\delta}$ and $Y_{1,\delta}$, if necessary, we may
assume that  they are disjoint. Finally, set
$\Y_\alpha=\Y_{\delta+1}=\Y_\delta\cup\{Y_{0,\delta},Y_{1,\delta}\}$
and note that $(3),(4),$ and $(5)$ are satisfied  for $\delta+1$
instead of $\delta$.

Next, we pass to the construction of $U_\delta, G_\delta$ satisfying
$(6)$. This will require that if $x\in U_\delta$ and $Y\in\Y_\delta$
converges to $x$, then $Y\subset^* U_\delta$. Let $G_\delta\in\G$ be
such that $|G_\delta\cap I_Y|<\w$ for all $Y\in\Y_\alpha$. Then
$F_\delta:=\bigcup_{n\in G_\delta}F^\delta_n$ is almost disjoint
from any $Y\in\Y_\alpha$. We need the following auxiliary

\begin{claim} \label{aux1_0}
For every $n\in\w$ there exists $C_n\subset\w$ such that
\begin{itemize}
\item[$(8)$] $Y\subset^*C_n$ for any $Y\in\Y_\alpha$ converging to $n$ in
$\la\w,\tau_\delta\ra$;
\item[$(9)$] $C_n\cap C_m=\emptyset$ for all $n\neq
m$;
\item[$(10)$] $C_n\cap F_\delta=\emptyset$ for all $n\neq
m$.
\end{itemize}
\end{claim}
\begin{proof}
For every $n\in\w$ let us denote by $\Y_{\alpha,n}$ the family of
all $Y\in\Y_\alpha$ converging to $n$, and fix a family
$\{O^n_k:k\in\w\setminus\{n\}\}\subset\tau_0$ such that $O^n_k$ is a
clopen neighbourhood of $k$ not containing $n$. Then $|Y\cap
O^n_k|<\w$ for all $Y\in\Y_{\alpha,n}$ and $k\neq n$, and hence
Lemma~\ref{no_w_<b_gaps} implies that there exists $C^0_n\subset\w$
such that $Y\subset^*C^0_n$ and $|C^0_n\cap O^n_k|<\w$ for all $Y,k$
as above. Let us note that $|C^0_n\cap Y|<\w$ for all
$Y\in\Y_\alpha\setminus \Y_{\alpha,n}$ because $Y\subset^* O^n_k$
for $k\neq n $ being the limit point of $Y$. Thus letting
$C_0=C^0_0\setminus F_\delta$ and
$C_n=C^0_n\setminus(\bigcup_{n'<n}C_{n'}\cup F_\delta)$ we get that
$\{C_n:n\in\w\}$ is a disjoint family satisfying $(8)-(10)$.
\end{proof}

The construction of $V_\delta:=\w\setminus U_\delta$  will be done
recursively over $k\in\w$, namely it will be constructed  as an
increasing union $\bigcup_{k\in\w}V^{\delta}_k$. Moreover, we shall
make sure that adding its complement to $\tau_\delta$ can also be
done without any harm. More precisely, set
$V^\delta_{-1}=\emptyset$, $V^\delta_0=F_\delta$ and assuming that
$V^\delta_k$ is constructed we let
$$V^\delta_{k+1}=V^\delta_k\cup\bigcup\{C_n\setminus\{\min C_n\}:n\in V^\delta_k\setminus V^\delta_{k-1}\}.$$
In the sequel we call a subset $B$ of $\w$ \emph{saturated} if $n\in
B$ implies $C_n\subset^* B$. It follows that $V_\delta$ defined
above is saturated. We claim that $U_\delta:=\w\setminus V_\delta$
is saturated as well. Indeed, otherwise there exists $n\in U_\delta$
such that $C_n\cap V_\delta\neq\emptyset$. Let $k\in\w$ be the
minimal such that $C_n\cap V^\delta_{k+1}\neq\emptyset$ (note that
$C_n\cap V^\delta_0=C_n\cap F_\delta=\emptyset$ by $(9)$). Then
there exists $m\in V^\delta_k\setminus V^\delta_{k-1}$ with $C_n\cap
C_m\neq\emptyset$, which is impossible since $n\neq m$ (because
$n\in U_\delta$ and $m\in V_\delta$).

Finally, $U_\delta\neq\emptyset$ since $\min C_n\in U_\delta$ for
all $n\in V_\delta.$

 Let $\tau_\alpha$ be the topology
generated by $\tau_\delta\cup\{U_\delta,V_\delta\}$  as a base.
 This way we get a
$0$-dimensional topology satisfying $(6)$. Since both
$U_\alpha,V_\alpha$ are saturated, it is easy to see that all $Y\in
\Y_\alpha$ are convergent also in $\la\w,\tau_\alpha\ra$. Indeed,
suppose that $y$ is the limit of $Y\in\Y_\alpha$ and $y\in U\cap
U_\delta$, where $U\in\tau_\delta$. Then $Y\subset^*U$ (because $Y$
converges in $\la\w,\tau_\delta\ra$) and $Y\subset^* C_y\subset^*
U_\delta$, and therefore $Y\subset^* U\cap U_\delta$. The same
argument works also for $V_\delta$ instead of $U_\delta$. Since
$U\in\tau_\delta$ was arbitrary, we conclude that $Y$ converges to
$y$ also in $\la\w,\tau_\alpha\ra$.

 To see that also $(7)$
holds, let us fix $n\in\w$ and $y\in U$, where $U\in \tau_\alpha$.
Let $Y\in \Y_0\cap[E_n]^\w \subset\Y_\alpha$ be convergent to $y$.
Then $Y\subset^* U$, and hence $Y\subset^*U\cap E_n$, which yields
$U\cap E_n\neq\emptyset$.

This completes our construction of the objects mentioned in
$(1)$-$(7)$ for all $ \delta<\hot c$, so that these conditions are
satisfied. Condition $(4)$  implies that $X=\la\w,\tau_{\hot c}\ra$
 is Frechet-Urysohn and $\alpha_4$. Indeed, if $\A$ is a countable
 family of mutually disjoint sequences convergent to some $x\in\w$
 in $X=\la\w,\tau_{\hot c}\ra$, then
 there exists $\delta<\hot c$ such that
 $\la \w\setminus\{x\},\A,x\ra=\la S_\delta,\A_\delta,x_\delta\ra$.
 Then $Y_{1,\delta}\in\Y_{\delta+1}$ is convergent to $x_\alpha$
 in $\la\w,\tau_{\hot c}\ra$, and it intersects infinitely many
 elements of $\A_\delta=\A$. Similarly one can check the
 Frechet-Urysohn property. Finally, $(6)$ implies the failure of the
 $ H$-separability, which completes the proof.
\end{proof}

The assumption of being $\alpha_4$
cannot be much improved in Theorem~\ref{p=c}
 as the following fact shows.

\begin{proposition} \label{alpha_2}
Every separable Frechet-Urysohn $\alpha_2$ space $X$ is
$H$-se\-pa\-ra\-ble.
\end{proposition}
\begin{proof}
Let $\la E_n:n\in\w\ra$ be a sequence of dense subsets of $X$, and
$D$ be a countable dense subset of $X$. Let $\{d_i:i\in\w\}$ be an
enumeration of all non-isolated elements of $D$, and
$\{d'_i:i\in\w\}=D\setminus \{d_i:i\in\w\}$. For every $i,n\in\w$
let us fix an injective sequence $S^i_n$ of elements of
$E_n\setminus\{d_i\}$ convergent to $d_i$. Applying $\alpha_2$ we
can find a sequence $\la s^i_n:n\in\w\ra$ convergent to $d_i$ such
that $s^i_n\in S^i_n\subset E_n$. Let us note that $\{d'_i:i\in\w\}$
consists of isolated points of $X$ and hence it is a subset of $E_n$
for all $n\in\w$. It follows that the sequence
$$ \big\la F_n:=\{d'_i:i\leq n\}\cup\{s^i_n:i\leq n\}\: :\: n\in\w  \big\ra $$
is a witness for the $H$-separability of $X$. Indeed, given any open
$U\subset X$, either there exists $i\in\w$ such that $d'_i\in U$,
and thus $U\cap F_n\neq\emptyset$ for all $n\geq i$; or there exists
$i\in\w$ with $d_i\in U$, and then $U\cap F_n\neq\emptyset$ as soon
as $s^i_n\in U$ and $n\geq i$, and $s^i_n\in U$ for all but finitely
many $n$ since $\la s^i_n:n\in\w\ra$ convergent to $d_i$.
\end{proof}

\section{Products of Fr\'echet-Urysohn spaces and $M$-separability}
\label{sec4}

In \cite{BarDow12}, the authors proved the next important result.
\begin{theorem}\label{BarDow}
(PFA) The product of finitely many countable FU spaces is $M$-separable.
\end{theorem}
It is worth noting that the proof given by the authors provides a stronger result: their argument proves that the product of finitely many countable FU space under PFA is in fact $mR$-separable.\\
\smallskip

The aim of this section is to show that MA is not sufficent in Theorem \ref{BarDow}. Namely we shall define a new set-theoretic principle which is consistent with MA, and, when combined with it, it implies the negation of the conclusion of Theorem \ref{BarDow}.\\
	\smallskip
	
Let $\tau_0$ be a topology on $\w$ turning it into a space
homeomorphic to the rationals $\mathbb Q$. In what follows
$(\ast_{\mathbb Q})$
 stands for the following statement:
\begin{quote}
 MA holds and there exists a mad family $\A$ on $\w$ such that
 \begin{itemize}
 \item every disjoint pair $\A',\A''\in [\A]^{<\hot c}$ is \emph{separated}, i.e. there exists $S\subset\w$
 such that $A'\subset^*S$ and $A''\subset^*\w\setminus S$ for any
 $A'\in\A'$ and $A''\in\A''$;
 \item Every $A\in\A$ is either closed discrete or a convergent sequence in $\la\w,\tau_0\ra$.
\end{itemize}
\end{quote}
The formally weaker statement $(\ast)$ obtained from $(\ast_{\mathbb Q})$ by dropping the second item is known to be consistent.
More precisely, the following result was proved by Dow and Shelah.
\begin{theorem}\cite{DowShe12}
It is consistent with MA and $\mathfrak{c} = \omega_{2}$ that there is a maximal almost disjoint family such that any disjoint pair of its $\leq \omega_{1}$-sized subsets are separated.
\end{theorem}

To prove the "whole statement", namely that $(\ast_{\mathbb Q})$ is consistent, first we remind the definitions of standard posets that will be involved.
Every filter $\F$ gives rise  to a natural forcing notion $\mathbb M_\F$
introducing a generic subset $X\in [\w]^\w$ such that $X\subset^* F$
for all $F\in\F$ as follows: $\mathbb M_\F$ consists of pairs
$\la s,F\ra$ such that $s\in [\w]^{<\w}$, $F\in\F$, and $\max s<\min F$.
A condition $\la s,F\ra$ is stronger than $\la t,G\ra$
if $F\subset G$, $s$ is an end-extension of $t$, and
$s\setminus t\subset G$. $\mathbb M_\F$ is usually called \emph{Mathias forcing
associated with} $\F$.

Every infinite almost disjoint family $\A\subset [\w]^\w$ generates
the dual filter $\F(\A)$ consisting of sets $F$ such that there
exists a finite $\A'\subset\A$ with the property
$\w\setminus\cup\A'\subset^*F$. We shall be interested in Mathias
forcing  associated with filters which are usually larger than
$\F(\A)$ for certain almost disjoint families $\A$ defined as
follows. Recall that $\tau_{0}$ is a topology on $\w$ such that
$\la\w,\tau_{0}\ra$ is homeomorphic to the rationals in $[0,1]$, and let
$\theta:\w\to\mathbb Q\cap [0,1]$ be the corresponding
homeomorphism. In what follows we shall refer to the topology $\tau_{0}$
whenever we speak about topological properties of subsets of $\w$,
e.g.,   convergence,  being closed discrete, etc.

Let us fix an almost disjoint family $\A\subset[\w]^\w$  consisting
of convergent sequences as well as closed discrete subsets.
Given any $n\in\w$, denote by $\F_n(\A)$ the filter generated by $\F(\A)\cup\{O\in\tau_{0}:n\in O\}$ and
note that if $G$ is $\mathbb M_{\F_n(A)}$-generic, then
$X_G:=\bigcup_{\la s,F\ra\in G}s$ is a sequence convergent to $n$,
and $|X_G\cap Y|=\w$ for any sequence $Y\in V$ convergent to $n$.

 Fix   $r\in [0,1]\setminus\mathbb
Q$ and let $\F_r(\A)$ be the filter generated by
$$\F(\A)\cup\{\theta^{-1}[O\cap\mathbb Q] \: :\: r\in O\mbox{ is an open subset of }\mathbb R \}.$$
By genericity, if $G$ is $\mathbb M_{\F_r(A)}$-generic, then
$X_G:=\bigcup_{\la s,F\ra\in G}s\subset\w$ is  closed and discrete
because $\theta[X_G]$ converges to an irrational number $r\in
[0,1]$, and $|X_G\cap Y|=\w$ for any closed discrete  subset $Y$ of
$\w$ such that $r$ is a limit point of $\theta[Y]$.

Recall from \cite{DowShe12} that an almost disjoint family
$\A\subset [\w]^\w$ is called \emph{special}, if there exists
$c:[\A]^{<\w}\to\w$ and a linear ordering $<$ of $\A$ such that for
each $n\in\w$ and sequences
$$ \la B_0,B_1,\ldots,B_{n-1}\ra, \la C_0,C_1,\ldots,C_{n-1}\ra\in\A^n $$
increasing with respect to $<$, if
$$c\la
B_0,B_1,\ldots,B_{n-1}\ra=c\la C_0,C_1,\ldots,C_{n-1}\ra=k,$$
then
for all $i\neq j$, $i,j\in n$ we have $B_i\cap C_j\subset k$.
There is a natural poset introduced in \cite[Def.~2.4]{DowShe12} and
denoted by $\IQ_{\A,<},$ which for an almost disjoint family $\A$
and a linear order $<$ thereof,  introduces a function $c$ having
the above property: A condition in this poset is a
 finite function $d:\mathcal P(\A_d)\to\w$, where $\A_d\in[\A]^{<\w}$, such that
  for each $n\in\w$ and $<$-increasing
sequences
$$ \la B_0,B_1,\ldots,B_{n-1}\ra, \la C_0,C_1,\ldots,C_{n-1}\ra\in\A_d^n, $$
if $d\la B_0,B_1,\ldots,B_{n-1}\ra=d\la
C_0,C_1,\ldots,C_{n-1}\ra=k,$ then for all $i\neq j$, $i,j\in n$ we
have $B_i\cap C_j\subset k$.
It is clear that $\A$ is special in $V^{\IQ_{\A,<}}$. However, this
poset may collapse cardinals: In ZFC one can construct an almost
disjoint family $\A$ of size $\w_1$ which cannot be made special by
any forcing which preserves $\w_1$, see \cite[p.~108]{DowShe12} and
references therein.

An almost disjoint family $\A\subset [\w]^\w$ is called
\emph{$\w_1$-special}, if any $\A'\in [\A]^{\w_1}$ is special. The
next fact is \cite[Prop.~1.5]{DowShe12}. It has been inspired by
earlier results of Silver, as it is stated in \cite{DowShe12}.
\begin{lemma}\label{silver}
If $\mathit{MA}_{\w_1}$ holds and $\A\in [[\w]^\w]^{\w_1}$ is $\w_1$-special, then
$\A$ is separated, i.e., for any $\A'\subset\A$ there exists $X\subset\w$ such that
$A\subset^*X$ for all $A\in\A'$ and $|A\cap X|<\w$ for all $A\in\A\setminus\A'$.
\end{lemma}
We are in a position now to present the sketch of the proof of the following
\begin{theorem} \label{(*Q)}
$(*_{\mathbb Q})$ is consistent.
\end{theorem}
\begin{proof}
We shall present only the main ideas of the proof,  the remaining details are exactly the same as in the proof of
\cite[Theorem~2.6]{DowShe12}.  Let us denote by $\Lambda$ the set of all limit ordinals below $\w_2$.
Assuming GCH in $V$, we shall construct
a finitely supported iteration
$\la\IP_\alpha,\name{\IQ}_\beta:\alpha\leq\w_2,\beta<\w_2\ra$  of
ccc posets of size $\w_1$, along with a sequence
$\la\name{\A}_\beta:\beta<\w_2, \beta\in\Lambda\ra$, where
$\name{\A_\beta}=\{\name{X}_\xi:\xi\in(\beta+1)\cap\Lambda\}$ is a $\IP_{\beta+1}$-name for an almost disjoint family,
  as follows:
\begin{itemize}
\item If $\beta$ is a successor ordinal, then $\name{\IQ}_\beta$ is  ($\IP_\beta$-name for) a ccc poset of size
$\w_1$;
\item For $\beta\in\Lambda$ the poset $\name{\IQ}_\beta$ is a $\IP_\beta$-name for the product
$\name{\IQ}^0_\beta\times\name{\IQ}^1_\beta$, where
     \begin{itemize}
        \item[$(a)$] Letting $\name{\A}^-_\beta$ be a $\IP_\beta$-name for
$\bigcup_{\xi\in\beta\cap \Lambda}\name{\A}_\xi=\{\name{X}_\xi:\xi\in\beta\cap\Lambda\}$,
$\name{\IQ}^0_\beta$ is either $\mathbb M_{\F_{\name{r}_\beta}(\name{\A}^-_\beta)}$ for some
$\IP_\beta$-name $\name{r}_\beta$ for an irrational number in $[0,1]$, or
$\mathbb M_{\F_{n_\beta}(\name{\A}^-_\beta)}$ for some
 $n_\beta\in\w$, decided by a bookkeeping function as described below;
$\name{\IQ}^0_\beta$  produces a generic subset $\name{X}_\beta$ almost disjoint from all elements of
$\name{\A}^-_\beta$, and we set $\name{\A}_\beta=\name{\A}^-_\beta\cup\{\name{X}_\beta\}$;
        \item[$(b)$]
        $\name{\IQ}^1_\beta=\name{\IQ}_{\name{\A}^-_\beta,<}$,
        where $<$ is the linear order (actually  wellorder) of $\name{\A}^-_\beta$
        generated by the indices of its elements, i.e.,
        $\name{X}_\xi<\name{X}_{\zeta}$ iff $\xi<\zeta,$
        $\xi,\zeta\in\beta\cap\Lambda$.
    \end{itemize}
\end{itemize}
Almost literally following the argument in the proof of
\cite[Th.~2.6]{DowShe12}, we can show that $\IP_{\w_2}$ is ccc, the
only non-trivial ingredient being that $\name{\IQ}^1_\beta$ is
forced to be ccc by $1_{\IP_\beta}$. This, roughly speaking, happens
because elements of $\name{\A}^-_\beta$ are added as generic reals
for corresponding Mathias posets. Now a standard choice of a
suitable bookkeeping function ``delivering'' the names for posets
$\name{\IQ}_\beta$ at successor steps $\beta$, and $n_\beta$ or
$\name{r}_\beta$ for limit $\beta$, guarantees that $(*_{\mathbb
Q})$ holds in $V^{\IP_{\w_2}}.$ More precisely, MA is guaranteed at
successor stages, and $\A=\bigcup_{\beta\in\Lambda}\A_\beta$ is a
mad family consisting of convergent sequences (added at stages
$\beta\in \Lambda$ for which the bookkeeping function gives
$n_\beta\in\w$) and closed discrete sets (added at stages $\beta\in
\Lambda$ for which the bookkeeping function gives $\name{r}_\beta$);
The presence of $\name{\IQ}^1_\beta$'s for $\beta\in\Lambda$ implies
that any $\A'\in [\A]^{\w_1}$ is special, which together with
Lemma~\ref{silver} implies that any such $\A'$ is separated, and
thus completes the proof of the theorem.
\end{proof}
\noindent\textbf{Remark.} \ Even though the main part of our proof
of Theorem~\ref{(*Q)}, namely that   $\name{\IQ}^1_\beta$ is forced
to be ccc by $1_{\IP_\beta}$, can be done in exactly the same way as
the corresponding step in the proof of \cite[Th.~2.6]{DowShe12}, the
overall scheme of the proof of Theorem~\ref{(*Q)} is very different
from that of \cite[Th.~2.1]{DowShe12}. More precisely, the proof of
the latter gives a so-called \emph{tight} mad family, i.e., a mad
family $\A$ such that for every countable $\mathcal
X\subset\I(\A)^+$ (recall that $X\in\I(\A)^+$ iff $X$ has infinite
intersection with infinitely many elements of $\A$) there exists
$A\in\A$ such that $|A\cap X|=\w$ for all $X\in\mathcal X$. On the
other hand, no mad family $\A$ of subsets of $\mathbb Q$ consisting
of convergent sequences and closed discrete sets can be tight.
Indeed, it is an easy exercise to prove that for every
$0<\delta<\epsilon$ there exists a closed discrete $D\subset\IQ$
such that $D\in\I(\A)^+$ and $\delta<|x|<\epsilon$ for all $x\in D$.
Now let $\la D_n:n\in\w\ra$ be a sequence of closed discrete subsets
of $\mathbb Q$ which lie in $\I(\A)^+$ and such that
$\frac{3}{5\cdot 2^n}<|x|<\frac{4}{5\cdot 2^n}$ for all $x\in D_n$.
Then it is clear that there is neither a closed discrete subset of
$\IQ$ nor a convergent sequence there which has infinite
intersection with each (even infinitely many) of $D_n$'s. \hfill
$\Box$
\medskip

%
%


The overall scheme of the proof of the next theorem is patterned after that of
\cite[Theorem~2.24]{BarDow11}. However, there are also essential  differences because our construction
lasts $>\w_1$-many steps and hence
we have to make sure that we do not face a kind of Hausdorff gaps
consisting of convergent sequences
 in our construction, which is done with the help of a mad family witnessing
 $(\ast_{\mathbb Q})$.

\begin{theorem}\label{ma_weaker_pfa}
$(\ast_{\mathbb Q})$ There exist two countable Frechet-Urysohn spaces $X_0,X_1$ without isolated points whose product is not $M$-separable.
\end{theorem}
\begin{proof}
Let the underlying sets of $X_0,X_1$ be $\w$ and $\tau_0=\sigma_0$ be a topology on $\w$
such that $\la \w,\tau_0\ra$ is homeomorphic to the rationals.
Let  $\A$ be a mad family on $\w$ whose existence is guaranteed by $(\ast_{\mathbb Q})$.

The final topologies $\tau=\tau_{\hot c}$ and $\sigma=\sigma_{\hot
c}$ turning $\w$ into Frechet-Urysohn-spaces with non-$M$-separable
product will be constructed recursively over ordinals $\alpha\in\hot
c$ as increasing unions $\tau_{\hot c}=\bigcup_{\alpha<\hot
c}\tau_\alpha$ and $\sigma_{\hot c}=\bigcup_{\alpha<\hot
c}\sigma_\alpha$ of intermediate zero-dimensional topologies. Let
$\la E_n:n\in\w\ra$ be a sequence of mutually disjoint dense subsets
of $\la \w\times\w,\tau_0\otimes\sigma_0\ra$ such that
$E:=\bigcup_{n\in\w}E_n$ is (the graph of) a permutation of $\w$.
 We shall make sure that $E_n$'s remain dense in $\la \w\times\w,\tau_\alpha\otimes\sigma_\alpha\ra$ for all $\alpha\leq\hot c$.
Let $\{\la S_\alpha,x_{\alpha}\ra:\alpha<\hot c\}$ be an enumeration of $[\w]^\w\times\w$ such that
for each $\la S,x\ra\in [\w]^\w\times\w$ there are cofinally many even as well as odd\footnote{An ordinal is \emph{even}
(resp. \emph{odd})
if it can be written in the form $\w\cdot\alpha+i$ for some even (resp. odd) $i\in\w$. } ordinals
$\alpha$ such that $\la S,x\ra=\la S_\alpha,x_\alpha\ra$.
Fix also an enumeration $\{\la F^\alpha_n:n\in\w\ra:\alpha<\hot c\}$ of $\prod_{n\in\w}[E_n]^{<\w}$.
For the convenience of those readers familiar with the proof of
 \cite[Theorem~2.24]{BarDow11}, for a subset $A$ of $\w$ we denote
$E[A]$ and $E^{-1}[A]$ (here we view $E$ as a map from $\w$ to $\w$)
by $E(A,1)$ and $E(A,0)$, respectively.

Suppose that for some $\alpha<\hot c$ and all  $\delta\in\alpha$ we
have already constructed topologies $\tau_\delta,\sigma_\delta$ on
$\w$, almost disjoint families $\Y_\delta, \Z_\delta\subset
[\w]^\w$,
 such that the following conditions are satisfied:
\begin{itemize}
\item[$(1)$] The weight of $\tau_\delta,\sigma_\delta$ is $<\hot c$;
\item[$(2)$]
\begin{itemize}
\item[$(a)$] $\Y_\beta\subset \Y_\delta$ for all $\beta\leq\delta$
and $\Y_\delta $ consists of sequences convergent in $\la\w,\tau_\delta\ra$;
\item[$(b)$] $\Z_\beta\subset \Z_\delta$ for all $\beta\leq\delta$
and $\Z_\delta $ consists of sequences convergent in $\la\w,\sigma_\delta\ra$;
\end{itemize}
\item[$(3)$]
\begin{itemize}
\item[$(a)$]
For every $\beta<\delta$, if $\beta$ is even and $x_\beta$ is a
limit point of $S_\beta$ in $\la\w,\tau_\delta\ra$, then there
exists $Y\in\Y_\delta $ converging to $x_\beta$ in
$\la\w,\tau_\delta\ra$ with $|Y\cap S_\beta|=\w$;
\item[$(b)$]
For every $\beta<\delta$, if $\beta$ is odd and $x_\beta$ is a limit
point of $S_\beta$  in $\la\w,\sigma_\delta\ra$, then there exists
$Z\in\Z_\delta$ converging to $x_\beta$ in $\la\w,\sigma_\delta\ra$
and such that $|Z\cap S_\beta|=\w$;
\end{itemize}
\item[$(4)$]
\begin{itemize}
\item[$(a)$] There are injective
 maps $\phi^\delta_0,\phi^\delta_1:\Y_\delta\to\A$
such that $Y\subset\phi^\delta_0(Y)$  and $E(Y,1)\subset
\phi^\delta_1(Y)$ for all $Y\in\Y_\delta$;
\item[$(b)$] There are injective maps $\psi^\delta_1,\psi^\delta_0:\Z_\delta\to\A$
such  that $Z\subset\psi^\delta_1(Y)$ and $E(Z,0)\subset
\psi^\delta_0(Z)$ for all $Z\in\Z_\delta$;
\item[$(c)$]
$\psi^\delta_1[\Z_\delta]\cap\phi^\delta_1[\Y_\delta]=\emptyset$ and
$\psi^\delta_0[\Z_\delta]\cap\phi^\delta_0[\Y_\delta]=\emptyset$;
\end{itemize}
\item[$(5)$]
\begin{itemize}
\item[$(a)$]
$E(\phi^\delta_0(Y),1)$ and $\phi^\delta_1(Y)$ are closed and
discrete in $\la\w,\sigma_\delta\ra$ for all  $Y\in\Y_\delta$;
\item[$(b)$]
$E(\psi^\delta_1(Z),0)$ and $\psi^\delta_0(Z)$ are closed and
discrete in $\la\w,\tau_\delta\ra$ for all  $Z\in\Z_\delta$;
\end{itemize}
\item[$(6)$] For every $\beta<\delta$ there are $U_\beta\in\tau_\delta$ and
 $V_\beta\in\sigma_\delta$ such that\footnote{In particular, $U_\beta,V_\beta,I_\beta$ depend only on $\beta$ and not on $\delta$.}
 $$ (U_\beta\times V_\beta)\cap\bigcup_{n\in \w}F^\beta_n=\emptyset. $$
\item[$(7)$]
$E_n$ is dense in $\la \w\times\w,\tau_\delta\otimes\sigma_\delta\ra$
for all $n\in\w$.
\end{itemize}
We shall consider several cases.

\textbf{\textit{I}}. $\alpha$ is limit. \\
It is easily checked that topologies $\tau_\alpha$ and $\sigma_\alpha$
generated by $\bigcup_{\delta<\alpha}\tau_\delta$
and $\bigcup_{\delta<\alpha}\sigma_\delta$ as a base, respectively,
along with $\Y_\alpha=\bigcup_{\delta<\alpha}\Y_\delta$
and $\Z_\alpha=\bigcup_{\delta<\alpha}\Z_\delta$ satisfy
 $(1)$-$(7)$
for  $\delta=\alpha$.

\textbf{\textit{II}}.  $\alpha=\delta+1$, $\delta$ is even,   $x_\delta$ is a
limit point of $S_\delta$ in $\la\w,\tau_\delta\ra$,
and $|Y\cap S_\delta|<\w$ for all
$Y\in\Y_\delta$ converging to $x_\delta$.   \\
 Since the weight of
$\la\w,\tau_\delta\ra$ is $<\hot c=\hot p$, there exists a sequence
$T\in [S_\delta]^\w$ convergent to $x_\delta$. Let $A^\delta_0\in\A$
be such that $|A^\delta_0\cap T|=\w$.

Analogously, since  $|\psi^\delta_0(Z)\cap T|<\w$ for all
$Z\in\Z_\delta$ (because $\psi^\delta_0(Z)$ is closed discrete in
$\la\w,\tau_\delta\ra$), we have  $\psi^\delta_0(Z)\neq A^\delta_0$
for all $Z\in\Z_\delta$. For the same reason
$|E(\psi^\delta_1(Z),0)\cap T|<\w$ for all  $Z\in\Z_\delta$. Since
$E$ is a permutation of $\w$, $\A_{1\downarrow
0}:=\{E(A,0):A\in\A\}$ is a mad family, and hence there exists
$A^\delta_1\in \A$ such that $|E(A^\delta_1,0)\cap A^\delta_0\cap
T|=\w$. It follows from the above that $E(A^\delta_1,0)\neq
E(\psi^\delta_1(Z),0)$ for all $Z\in\Z_\delta$, and hence also
$A^\delta_1 \neq \psi^\delta_1(Z)$.

Finally, set $Y_\delta=T\cap A^\delta_0\cap E(A^\delta_1,0)$,
$\Y_\alpha:=\Y_\delta\cup\{Y_\delta\}$, $\Z_\alpha=\Z_\delta$,
$\phi^\alpha_0(Y_\delta)=A^\delta_0$,
$\phi^\alpha_1(Y_\delta)=A^\delta_1$,
$\phi^\alpha_0\uhr\Y_\delta=\phi^\delta_0$, and
$\phi^\alpha_1\uhr\Y_\delta=\phi^\delta_1$. It follows from the
construction that item $(4)$ is satisfied. To proceed further we
need the following direct consequence of
\cite[Lemma~2.23]{BarDow11}.
\begin{claim}\label{bardow} 
Suppose that $X$ is a countable crowded space of weight less than $\hot p$,
$\mathcal T \subset \mathcal P(X)$ is a family of almost disjoint
converging sequences in $X$, $|\mathcal T| < \hot p$, and $\mathcal R\subset\mathcal P
 (X)$ is a countable family of subsets of $X$ such that each $R\in\mathcal R$ has dense complement
 and is almost disjoint from each member of
$\mathcal T$, with all but possibly one elements of $\mathcal R$ being discrete.
 Then there is an expansion of the topology of $X$ to a larger crowded one, obtained by
adding countably many sets along with their complements, in which each $R\in\mathcal R$ is a
closed nowhere dense set, and each member of $\mathcal T$ is again a
converging sequence.
\end{claim}
For every $\la x,y\ra\in \w^2$  and $n\in\w$ find an injective
sequence
$$\big\la
\la s^{\la x,y\ra,n}_i,t^{\la x,y\ra,n}_i\ra:i\in\w\big\ra\in \big(E_n\cap \big((\w\setminus A^\delta_0)\times
(\w\setminus A^\delta_1)\big)^\w$$
 convergent to $\la x,y\ra$. This is possible because
 all elements of $\A$ are nowhere  dense in $\tau_0=\sigma_0$, and hence
 also in all further
 topologies we construct.
 It follows that $T^{x,y}:=\{t^{\la x,y\ra,n}_i:i\in\w\}$ is almost disjoint from
 each  $Z\in\Z_\alpha$. Indeed, $E(T^{\la x,y\ra,n},0)=\{s^{\la x,y\ra,n}_i:i\in\w\}$
  is a convergent sequence in $\la\w,\tau_\delta\ra$,
 whereas $E(Z,0)$ is closed discrete  in this space for all $Z\in\Z_\delta$. Also,
 $E(T^{\la x,y\ra,n},0)\cap A^\delta_0 =\emptyset$ by the choice of $T^{\la x,y\ra,n}$,
 and hence $T^{\la x,y\ra,n}\cap E(A^\delta_0,1) =\emptyset$ for all $x,y$ as above.
 In particular, since $y\in\w$ is arbitrary, it shows that $\w\setminus E(A^\delta_0,1)$
 is dense\footnote{This is of course easy and follows directly from the fact that
 $A^\delta_0$ is nowhere dense in $\la\w,\tau_\delta\ra$.} in $\la\w,\sigma_\delta\ra$.
  Also, $T^{\la x,y\ra,n}$ is almost disjoint from $A^\delta_1$ for all $x,y$ by
construction. Applying Lemma~\ref{bardow}  to $X=\la
\w,\sigma_\delta\ra$, $\mathcal
T=\Z_\delta\cup\{T^{\la x,y\ra,n}:x,y,n\in\w\},$ $\mathcal R=\{
E(A^\delta_0,1),A^{\delta}_1\}$, we get a $0$-dimensional topology
$\sigma^{-}_\alpha\supset\sigma_\delta$ on $\w$ in which  elements of
$\mathcal T$ are converging, and those of $\mathcal R$ are closed
discrete. By the choice of $\mathcal T$ we have that $E_n$ is dense
in $\la\w\times\w,\tau_\delta\otimes\sigma^{-}_\alpha\ra$ for all $n$,
because there is a sequence in $E_n$ converging to each $\la
x,y\ra$. Thus condition $(5)$ will be satisfied because the topology
$\sigma_\alpha$ we plan to construct is going to be crowded and
stronger than $\sigma^-_\alpha$. For convenience we set
$\tau^-_\alpha=\tau_\delta$.

Next, we pass to the construction of $U_\delta,V_\delta$ satisfying
$(6)$. This condition requires that if $x\in U_\delta$ and
$Y\in\Y_\delta$ converges to $x$, then $Y\subset^* U_\delta$. To
achieve this we shall use the following auxiliary

\begin{claim} \label{aux1}
There are families $\{C_n:n\in\w\}$ and $\{D_n:n\in\w\}$
 such that
\begin{itemize}
\item[$(8)$] $Y\subset^*C_n$ for any $Y\in\Y_\alpha$ converging to $n$ in
$\la\w,\tau^{-}_\alpha\ra$;
\item[$(9)$] $Z\subset^*D_n$ for any $Z\in\Z_\alpha$ converging to $n$ in
$\la\w,\sigma^{-}_\alpha\ra$;
\item[$(10)$] $C_n\cap C_m=\emptyset$ and $D_n\cap D_m=\emptyset$ for all $n\neq
m$;
\item[$(11)$] $E(C_n,1)\cap D_m=\emptyset$
(or equivalently, $E(D_m,0)\cap C_n=\emptyset$)
 for any $n,m\in\w$.
\end{itemize}
\end{claim}
\begin{proof}

For every $n\in\w$ let us denote by $\Y_{\alpha,n}$ the family
of all $Y\in\Y_\alpha$ converging to $n$, and fix a family $\{O^n_k:k\in\w\setminus\{n\}\}\subset\tau_0$
such that $O^n_k$ is a clopen neighbourhood of $k$ not containing $n$. Then $|Y\cap O^n_k|<\w$
for all $Y\in\Y_{\alpha,n}$ and $k\neq n$, and hence Lemma~\ref{no_w_<b_gaps} implies that there exists $C^0_n\subset\w$
such that $Y\subset^*C^0_n$ and $|C^0_n\cap O^n_k|<\w$ for all  $Y,k$ as above.
Let us note that $|C^0_n\cap Y|<\w$ for all $Y\in\Y_\alpha\setminus \Y_{\alpha,n}$
because $Y\subset^* O^n_k$ for $k $ being the limit point of $Y$.
Thus letting $C^1_0=C^0_0$ and $C^1_n=C^0_n\setminus\bigcup_{n'<n}C^0_{n'}$
we get that $\{C^1_n:n\in\w\}$ is a disjoint family such that $Y\subset^* C^1_n$
for all $Y\in\Y_{\alpha,n}$. Similarly we can get a disjoint family $\{D^1_n:n\in\w\}$
   such that $Z\subset^* D^1_n$
for all $Z\in\Z_{\alpha,n}$, where $\Z_{\alpha,n}=\{Z\in\Z_\alpha:Z$
converges to $n\}$. Finally, by item $(4)$ and the property of $\A$
stated in $(\ast_{\mathbb Q})$ we can find $C,D\subset\w$ such that
$\phi^\alpha_0(Y)\subset^*C$ and $\phi^\alpha_1(Y)\subset^*D$ for
all $Y\in\Y_\alpha$, and $|\psi^\alpha_1(Z)\cap D|<\w$ and
 $|\psi^\alpha_0(Z)\cap C|<\w$
for all  $Z\in\Z_\alpha$.

Set $C_n=C^1_n\cap C\cap E(D,0)$ and $D_n=D^1_n\cap (\w\setminus D)\cap(\w\setminus E(C,1))$.
We claim that the families $\{C_n:n\in\w\}$
 and $\{D_n:n\in\w\}$ are as required.
 Indeed, for $(8)$ let us note that for all $Y\in\Y_{\alpha,n}$ we have
 $Y\subset^*C^1_n$,  $Y\subset \phi^\delta_0(Y)\subset^*C$,
 and $E(Y,1)\subset\phi^\alpha_1(Y)\subset^* D$, the latter implying $Y\subset^* E(D,0)$,
 and hence $Y\subset^* C_n$. $(9)$ is analogous:
 for all $Z\in\Z_{\alpha,n}$ we have
 $Z\subset^*D^1_n$,  $Z\subset \psi^\delta_1(Z)\subset^*\w\setminus D$,
 and $E(Z,0)\subset\psi^\alpha_0(Z)\subset^* \w\setminus C$, the
 latter implying $Z\subset^* \w\setminus E(C,1)$,
 and hence $Z\subset^* D_n$. Condition $(10)$ follows from
 $C_n\subset C^1_n$ and $D_n\subset D^1_n$.
 And finally,
 $$ E(C_n,1)\cap D_m\subset E\big(E(D,0),1\big)\cap (\w\setminus D)=D\cap (\w\setminus D)=\emptyset $$
 for all $n,m\in\w$, which yields $(11).$
\end{proof}

For every $\la x,y\ra\in\w^2$ and $n\in\w$ we
 fix an injective sequence
 $$\big\la \la s^{\la x,y\ra,n}_k,t^{\la x,y\ra,n}_k\ra:k\in\w \big\ra \in
(E_n\setminus F^\delta_n)^\w$$ converging to $\la x,y\ra$ in
$\la\w\times\w,\tau^-_\alpha\otimes\sigma^{-}_\alpha\ra$. We shall in
addition assume that these sequences are mutually disjoint, i.e.,
$ s^{\la x,y\ra,n}_k= s^{\la x',y'\ra,n'}_{k'}$ (resp. $ t^{\la x,y\ra,n}_k= t^{\la x',y'\ra,n'}_{k'}$) iff
$x=x'$, $y=y'$, $n=n'$, and $k=k'$.
Set $S^{\la x,y\ra,n}=\{s^{\la x,y\ra,n}_k:k\in\w\}$ and $T^{\la x,y\ra,n}=\{t^{\la x,y\ra,n}_k:k\in\w\}$ for all $\la x,y\ra$ and $n$ as above.

Replacing the $D_m$'s and $C_m$'s with smaller sets, if necessary, we may additionally assume that
$$|S^{\la x,y\ra,n}\cap C_m|<\w \mbox{ \ and\ \ } |T^{\la x,y\ra,n}\cap D_m|<\w$$
 for all $\la x,y\ra$, $n$, and $m$.
Indeed,
since $E(Y,1)$ is closed and discrete in $\la\w,\sigma^{-}_{\alpha}\ra$ for every $Y\in \Y_\alpha$
and $E(S^{\la x,y\ra,n},1)=T^{\la x,y\ra,n}$ is a convergent sequence in this topology,
we conclude that $|S^{\la x,y\ra,n} \cap Y|<\w$ for all $\la x,y\ra,n$ as above.
Since there are no $(\w,<\hot b)$ gaps, there exists $S\subset\w$ such that
$S^{\la x,y\ra,n}\subset^*S$ and $|Y\cap S|<\w$ for all $\la x,y\ra,n$. Thus replacing
$C_m$ with $C_m\setminus S$ for all $m\in\w$, if necessary, we may  assume that
$|S^{\la x,y\ra,n}\cap C_m|<\w$ for all $\la x,y\ra,n$, in addition to all properties
of $C_m$ stated above. Analogously with  $|T^{\la x,y\ra,n}\cap D_m|<\w$ for all $\la x,y\ra,n$.

Finally, we pass to the construction of $U_\delta$ and $V_
\delta$. This will be done
recursively over $k\in\w$, namely they will be constructed  as  increasing unions
$\bigcup_{k\in\w}U^{\delta,1}_k$  and
$\bigcup_{k\in\w}V^{\delta,1}_k$, respectively. Moreover, we shall make sure that
adding their complements to the corresponding topologies can also be
done without any harm, by constructing these as  increasing unions
$\bigcup_{k\in\w}U^{\delta,0}_k$ and
$\bigcup_{k\in\w}V^{\delta,0}_k$, respectively. We  construct these objects
along with
\begin{itemize}
\item non-decreasing sequences
$\la H^j_k:k\in\w\ra$ and $\la G^j_k:k\in\w\ra$ of finite subsets of $\w$, where $j\in 2$; and
\item for each $x\in\bigcup_{k\in\w}H^1_k\cup\bigcup_{k\in\w}H^0_k$,
$y\in\bigcup_{k\in\w}G^1_k\cup\bigcup_{k\in\w}G^0_k$ and $n\in\w$
cofinal subsets $\tilde{C}_x$, $\tilde{D}_y$,
 $\tilde{S}^{\la x,y\ra,n}$ and $\tilde{T}^{\la
x,y\ra,n}$ of $C_x$, $D_y$, $S^{\la x,y\ra,n}$ and $T^{\la
x,y\ra,n}$, respectively,
\end{itemize}
such that $H^1_0=\{x^1_0\}$ and $G^1_0=\{y^1_0\}$ for some $\la
x^1_0,y^1_0\ra\not\in\bigcup_{n\in\w}F^\delta_n$,
$H^0_0=G^0_0=\emptyset$, and
 the
following conditions are satisfied for all $k\in\w$ and $i\in 2$:
\begin{itemize}
\item[$(i)$] $H^1_k\cap H^0_k=G^1_k\cap G^0_k=\emptyset$;
\item[$(ii)$] $U^{\delta,i}_k=\big[H^i_k\cup\bigcup\{\tilde{C}_x:x\in H^i_k\}\cup \\
\cup \bigcup\{\tilde{S}^{\la x,y\ra,n}:n\leq k,
\la x,y\ra\in H^i_k\times(G^1_k\cup G^0_k)\} \big]\cup\Delta^i_k$,\\
$V^{\delta,i}_k=\big[G^i_k\cup\bigcup\{\tilde{D}_y:y\in G^i_k\}\cup\\
\cup\bigcup\{\tilde{T}^{\la x,y\ra,n}:n\leq k, \la x,y\ra\in
(H^1_k\cup H^0_k)\times G^i_k \} \big]\cup\Sigma^i_k$;
\item[$(iii)$] $\Delta^1_k=\emptyset$ and $\Delta^0_k=\{k\}$ if
$k\not\in U^{\delta,1}_k$, otherwise $\Delta^0_k=\emptyset$;
$\Sigma^1_k=\emptyset$ and $\Sigma^0_k=\{k\}$ if $k\not\in
V^{\delta,1}_k$, otherwise $\Sigma^0_k=\emptyset$;
\item[$(iv)$] $U^{\delta,1}_k\cap U^{\delta,0}_k=V^{\delta,1}_k\cap
V^{\delta,0}_k=\emptyset$;
\item[$(v)$] If $\la x,y\ra\in
(H^1_k\cup H^0_k)\times (G^1_k\cup G^0_k)$, then $\tilde{T}^{\la
x,y\ra,n}= E(\tilde{S}^{\la x,y\ra,n},1)$ for all
 $n\leq k$;
\item[$(vi)$] $(U^{\delta,1}_k\times
 V^{\delta,1}_k)\cap\bigcup_{n\in\w}F^\delta_n=\emptyset$;
\item[$(vii)$] $H^i_{k+1}=H^i_k\cup\{\min(U^{\delta,i}_k\setminus H^i_k)\}$
provided that $H^i_k\neq\emptyset$, and $H^i_{k+1}=H^i_k=\emptyset $
otherwise;  $G^i_{k+1}=G^i_k\cup\{\min(V^{\delta,i}_k\setminus
G^i_k)\} $ provided that $G^i_k\neq\emptyset$, and
$G^i_{k+1}=G^i_k=\emptyset $ otherwise.
\end{itemize}
To start off the inductive construction, set
$\tilde{C}_{x^1_0}=C_{x^1_0}\setminus E(\{y^1_0\},0),$
$\tilde{D}_{y^1_0}=D_{y^1_0}\setminus E(\{x^1_0\},1),$
 $$\hat{T}^{\la
x^1_0,y^1_0\ra,0}=T^{\la x^1_0,y^1_0\ra,0}\setminus \big(
E(C_{x^1_0},1)\cup E(\{x^1_0\},1)\big),$$
$$\hat{S}^{\la
x^1_0,y^1_0\ra,0}=S^{\la x^1_0,y^1_0\ra,0}\setminus\big(
E(D_{y^1_0},0)\cup E(\{y^1_0\},0)\big)$$ and then
$$\tilde{T}^{\la x^1_0,y^1_0\ra,0}=\hat{T}^{\la
x^1_0,y^1_0\ra,0}\cap E(\hat{S}^{\la x^1_0,y^1_0\ra,0},1)$$ and
$$\tilde{S}^{\la x^1_0,y^1_0\ra,0}=\hat{S}^{\la
x^1_0,y^1_0\ra,0}\cap E(\hat{T}^{\la x^1_0,y^1_0\ra,0},0).$$ The
equation
\begin{eqnarray*}
E (\tilde{S}^{\la x^1_0,y^1_0\ra,0},1)=E\big(\hat{S}^{\la
x^1_0,y^1_0\ra,0}\cap E(\hat{T}^{\la x^1_0,y^1_0\ra,0},0),1\big)=\\
E(\hat{S}^{\la x^1_0,y^1_0\ra,0},1) \cap E\big(E(\hat{T}^{\la
x^1_0,y^1_0\ra,0},0),1\big)=\\
=E(\hat{S}^{\la x^1_0,y^1_0\ra,0},1)\cap \hat{T}^{\la
x^1_0,y^1_0\ra,0}=\tilde{T}^{\la x^1_0,y^1_0\ra,0}
\end{eqnarray*}
proves $(v)$ for $k=0$. To check  $(vi)$ note that
\begin{eqnarray*}
U^{0,1}_0\times
V^{0,1}_0=\big(\{x^1_0\}\cup\tilde{C}_{x^1_0}\cup\tilde{S}^{\la
x^1_0,y^1_0\ra,0}\big)\times
\big(\{y^1_0\}\cup\tilde{D}_{y^1_0}\cup\tilde{T}^{\la
x^1_0,y^1_0\ra,0}\big)=\\
\{\la x^1_0,y^1_0\ra\}\cup(\{x^1_0\}\times \tilde{D}_{y^1_0})\cup
(\{x^1_0\}\times\tilde{T}^{\la x^1_0,y^1_0\ra,0}) \cup\\
\cup
(\tilde{C}_{x^1_0}\times\{y^1_0\})\cup(\tilde{C}_{x^1_0}\times\tilde{D}_{y^1_0})\cup(\tilde{C}_{x^1_0}\times\tilde{T}^{\la
x^1_0,y^1_0\ra,0})\cup\\
\cup (\tilde{S}^{\la
x^1_0,y^1_0\ra,0}\times\{y^1_0\})\cup(\tilde{S}^{\la
x^1_0,y^1_0\ra,0}\times\tilde{D}_{y^1_0})\cup(\tilde{S}^{\la
x^1_0,y^1_0\ra,0}\times\tilde{T}^{\la x^1_0,y^1_0\ra,0}),
\end{eqnarray*}
i.e., $U^{0,1}_0\times V^{0,1}_0$  is a union of 9 sets, of which 7
ones ``in the middle'' have empty intersection with
$\bigcup_{n\in\w}E_n$ because $A\times
B\cap\bigcup_{n\in\w}E_n=\emptyset$ provided that $A\cap
E(B,0)=\emptyset$, which is the case for these 7 products by the
choice or definition of the corresponding sets. Also, $\la
x^1_0,y^1_0\ra\not\in\bigcup_{n\in\w}F^\delta_n$ by the choice,
whereas
$$ \tilde{S}^{\la
x^1_0,y^1_0\ra,0}\times\tilde{T}^{\la x^1_0,y^1_0\ra,0}\cap
\bigcup_{n\in\w}E_n=\{\la s^{\la x^1_0,y^1_0\ra,0}_k, t^{\la
x^1_0,y^1_0\ra,0}_k\ra:k\in\w\}\subset E_0\setminus F^\delta_0,
$$
and hence $\tilde{S}^{\la x^1_0,y^1_0\ra,0}\times\tilde{T}^{\la
x^1_0,y^1_0\ra,0} $ also has empty intersection with
$\bigcup_{n\in\w}F^\delta_n$. This
  completes the case $k=0$ since all
other items are easily checked.

 Suppose now that items $(i)-(vii)$
hold for $k$ and note that $(vii)$ gives the unique way to define
$H^{i}_{k+1}$ and $G^i_{k+1}$ for $i\in 2$, and $(iv)$ for $k$
yields $(i)$ for $k+1$. Given $i\in 2$ and $x\in H^i_{k+1}\setminus
H^i_k$ and $y\in G^i_{k+1}\setminus G^i_k$, set
\begin{eqnarray*}
\tilde{C}_x=C_x\setminus \Big( E(G^1_{k+1},0)\cup
\bigcup\big\{E(T^{\la x',y'\ra,n'}, 0)\::\\
 \la x',y'\ra\in
(H^1_{k+1}\cup H^0_{k+1})\times (G^1_{k+1}\cup G^0_{k+1}), n'\leq
k\big\}\cup H^0_{k+1}\Big)=\\
= C_x\setminus \Big( E(G^1_{k+1},0)\cup\bigcup\big\{S^{\la
x',y'\ra,n'}\::\\
 \la x',y'\ra\in (H^1_{k+1}\cup H^0_{k+1})\times
(G^1_{k+1}\cup G^0_{k+1}), n'\leq k\big\}\cup H^0_{k+1}\Big)
\end{eqnarray*}
and
\begin{eqnarray*}
\tilde{D}_y=D_y\setminus\Big( E(H^1_{k+1},1)
\cup\bigcup\big\{E(S^{\la x',y'\ra,n'}, 1)\:: \\ \la x',y'\ra\in
(H^1_{k+1}\cup H^0_{k+1})\times (G^1_{k+1}\cup G^0_{k+1}), n'\leq
k\big\}\cup G^0_{k+1}\Big)=\\
= D_y\setminus\Big( E(H^1_{k+1},1)\cup\bigcup\big\{T^{\la
x',y'\ra,n'}\:: \\
 \la x',y'\ra\in (H^1_{k+1}\cup H^0_{k+1})\times
(G^1_{k+1}\cup G^0_{k+1}), n'\leq k\big\}\cup G^0_{k+1} \Big).
\end{eqnarray*}
For every
\begin{eqnarray*}
\la\la x',y'\ra,n'\ra\in \big((H^1_{k+1}\cup H^0_{k+1})\times
(G^1_{k+1}\cup G^0_{k+1})\times (k+2)\big)\setminus \\
\setminus \big( (H^1_k\cup H^0_k)\times (G^1_k\cup G^0_k) \times
(k+1)\big)
\end{eqnarray*}
 we let
\begin{eqnarray*}
\hat{T}^{\la x',y'\ra,n'}=T^{\la x',y'\ra,n'}\setminus \big[
E\big(\bigcup\big\{
C_x:x\in H^{1}_{k+1} \big\}\cup H^1_{k+1}, 1\big)\cup G^0_{k+1}\cup\\
\cup\bigcup\big\{D_y:y\in G^1_{k+1}\cup G^0_{k+1}\big\}\big],\\
\hat{S}^{\la x',y'\ra,n'}=S^{\la x',y'\ra,n'}\setminus
\big[E\big(\bigcup\big\{D_y:y\in G^{1}_{k+1} \big\}\cup G^1_{k+1},
0\big) \cup H^0_{k+1} \cup\\
\cup\bigcup\big\{C_x:x\in H^1_{k+1}\cup H^0_{k+1}\big\}\big],
\end{eqnarray*}
and finally
\begin{eqnarray*}
\tilde{T}^{\la x',y'\ra,n'}:=\hat{T}^{\la x',y'\ra,n'}\cap
E(\hat{S}^{\la x',y'\ra,n'},1),\\
\tilde{S}^{\la x',y'\ra,n'}:=\hat{S}^{\la x',y'\ra,n'}\cap
E(\hat{T}^{\la x',y'\ra,n'},0).
\end{eqnarray*}
The equality
\begin{eqnarray*}
E (\tilde{S}^{\la x',y'\ra,n'},1)=E\big(\hat{S}^{\la x',y'\ra,n'}\cap E(\hat{T}^{\la x',y'\ra,n'},0),1\big)=\\
E(\hat{S}^{\la x',y'\ra,n'},1) \cap E\big(E(\hat{T}^{\la x',y'\ra,n'},0),1\big)=\\
=E(\hat{S}^{\la x',y'\ra,n'},1)\cap \hat{T}^{\la
x',y'\ra,n'}=\tilde{T}^{\la x',y'\ra,n'}
\end{eqnarray*}
for all $\la \la x',y'\ra,n'\ra$ as above  yields $(v)$ for $k+1$.
Let $U^{\delta,i}_{k+1}$ and $V^{\delta,i}_{k+1}$ be defined
according to $(ii)$ and $(iii)$ for $k+1$. To check $(vi)$ for
$(k+1)$,  let us consider the product\\
\[
\begin{array}{ll}
   &  U^{\delta,1}_{k+1}\times V^{\delta,1}_{k+1}=\\
(\alpha) &    = \big(H^1_{k+1}\times G^1_{k+1}\big)\cup
\big(H^1_{k+1}\times\bigcup\{\tilde{D}_y:y\in
G^1_{k+1}\}\big)\bigcup
\\
(\beta) & \bigcup\big(H^1_{k+1}\times\bigcup\{\tilde{T}^{\la
x,y\ra,n}:n\leq k+1,
\la x,y\ra\in \\
\ & \in (H^1_{k+1}\cup H^0_{k+1})\times G^1_{k+1} \}\big)\bigcup \\
 (\gamma) &  \bigcup  \big(\bigcup\{\tilde{C}_x:x\in H^1_{k+1}\}\times
G^1_{k+1}\big)\bigcup\\
(\zeta) & \bigcup
\big(\bigcup\{\tilde{C}_x:x\in H^1_{k+1}\}\times\bigcup\{\tilde{D}_y:y\in G^1_{k+1}\}\big)\bigcup\\
(\varepsilon) & \bigcup\big(\bigcup\{\tilde{C}_x:x\in
H^1_{k+1}\}\times\\
 \ & \times\bigcup\{\tilde{T}^{\la x,y\ra,n}:n\leq k+1,
\la x,y\ra\in (H^1_{k+1}\cup H^0_{k+1})\times G^1_{k+1} \}\big)\bigcup \\
 (\eta) & \bigcup\big(\bigcup\{\tilde{S}^{\la x,y\ra,n}:n\leq k+1, \la x,y\ra\in \\
 \ & \in H^1_{k+1}\times(G^1_{k+1}\cup G^0_{k+1})\}\times G^1_{k+1}\big)\bigcup\\
(\theta) & \bigcup \big(\bigcup\{\tilde{S}^{\la x,y\ra,n}:n\leq k+1,
\la x,y\ra\in H^1_{k+1}\times(G^1_{k+1}\cup G^0_{k+1})\}\times\\
\ & \times\bigcup\{\tilde{D}_y:y\in G^1_{k+1}\}\big)\bigcup \\
(\lambda) & \bigcup \big(\bigcup\{\tilde{S}^{\la x,y\ra,n}:n\leq
k+1, \la x,y\ra\in
H^1_{k+1}\times(G^1_{k+1}\cup G^0_{k+1})\}\times  \\
\ & \times\bigcup\{\tilde{T}^{\la x,y\ra,n}:n\leq k+1, \la x,y\ra\in
(H^1_{k+1}\cup H^0_{k+1})\times G^1_{k+1} \}\big).
\end{array}
\]
\\
To show that this product is disjoint from
$F=\bigcup_{n\in\w}F^\delta_n$ we shall analyze sets appearing in
the above long formula one by one. $H^{1}_{k+1}\times G^1_{k+1}$ is
disjoint from $F$ by our inductive assumption, namely $(vi)$ for
$k$, since $H^1_{k+1}\subset U^{\delta,1}_k$ and $G^1_{k+1}\subset
V^{\delta,1}_k$. The second product in $(\alpha)$ is disjoint  from
$F$  because we have explicitly made $E(H^1_{k+1},1)$ disjoint from
$\tilde{D}_y$ if $y\in G^1_{k+1}\setminus G^1_k$, and hence
$H^1_{k+1}\times\tilde{D}_y$ is disjoint even from $E$ in this case.
And if $y\in G^1_k$ then we can use $(vi)$ for $k$ since
$H^1_{k+1}\subset U^{\delta,1}_k$ and $\tilde{D}_y\subset
V^{\delta,1}_k$. Analogously we can show that also the product in
item $(\gamma)$ is disjoint from $F$.

The  product in $(\beta)$ is disjoint  from $F$ because we have
explicitly made $E(H^1_{k+1},1)$ disjoint from $\tilde{T}^{\la
x,y\ra,n}$ if $\la\la x,y\ra,n\ra$ lies in the difference $N:=$
$$ \big(((H^1_{k+1}\cup H^0_{k+1})\times G^1_{k+1})\times
(k+2)\big)\setminus \big(((H^1_{k}\cup H^0_{k})\times G^1_{k})\times
(k+1)\big),$$ and if
$$\la\la x,y\ra,n\ra\in((H^1_{k}\cup H^0_{k})\times G^1_{k})\times
(k+1)$$
 then we can use $(vi)$ for $k$ since
$H^1_{k+1}\subset U^{\delta,1}_k$ and $\tilde{T}^{\la
x,y\ra,n}\subset V^{\delta,1}_k$. Analogously we can show that also
the product in item $(\eta)$ is disjoint from $F$.

The product in $(\zeta)$ is  disjoint even from $E$ since
$E(C_x,1)\cap D_y=\emptyset$ for any $x,y\in\w$ by
Claim~\ref{aux1}$(11)$.

The  product in item $(\varepsilon)$ is disjoint  from $F$ because
we have explicitly made $E(\tilde{C}_x,1)$ disjoint from
$\tilde{T}^{\la x',y'\ra,n}$ if $x\in H^1_{k+1}\setminus H^1_k$ or
$\la\la x',y'\ra,n\ra\in N$,  and if both $x\in H^1_k$ and
$$\la\la x',y'\ra,n\ra\in(H^1_{k}\cup H^0_{k})\times (G^1_{k}\cup G^0_k)\times
(k+1)$$
 then we can use $(vi)$ for $k$ since
$\tilde{C}_x\subset U^{\delta,1}_k$ and $\tilde{T}^{\la
x',y'\ra,n}\subset V^{\delta,1}_k$. Analogously we can show that
also the product in item $(\theta)$ is disjoint from $F$.

Finally, consider the product in $(\lambda)$. Given  $\la\la
x,y\ra,n\ra$ and $\la\la x',y'\ra,n'\ra$ in
$$ ((H^1_{k+1}\cup H^0_{k+1})\times G^1_{k+1})\times
(k+2), $$ two cases are possible. If $\la\la x,y\ra,n\ra =\la\la
x',y'\ra,n'\ra$, then $\tilde{T}^{\la x,y\ra,n}=E(\tilde{S}^{\la
x,y\ra,n},1)$, and hence
\begin{eqnarray*}
(\tilde{S}^{\la x,y\ra,n}\times\tilde{T}^{\la x,y\ra,n})\bigcap
E\subset\{\la s^{\la x,y\ra,n}_k, t^{\la
x,y\ra,n}_k\ra:k\in\w\}\subset E_n\setminus F^\delta_n,
\end{eqnarray*}
i.e., $\tilde{S}^{\la x,y\ra,n}\times\tilde{T}^{\la x,y\ra,n}$ is
disjoint from $F$. If $\la\la x,y\ra,n\ra \neq \la\la
x',y'\ra,n'\ra$, then
$$E(\tilde{S}^{\la x,y\ra,n},1)= \tilde{T}^{\la
x,y\ra,n}\subset T^{\la x,y\ra,n},$$ and hence
$$E(\tilde{S}^{\la
x,y\ra,n},1)\cap \tilde{T}^{\la x',y'\ra,n'}\subset T^{\la
x,y\ra,n}\cap T^{\la x',y'\ra,n'}=\emptyset,$$ which implies that
$(\tilde{S}^{\la x,y\ra,n}\cap \tilde{T}^{\la x',y'\ra,n'})\cap
E=\emptyset$. This completes our proof of $(vi)$ for $(k+1)$.

Finally, it remains to check $(iv)$ for $k+1$. We can write
$U^{\delta,1}_{k+1}\cap U^{\delta,0}_{k+1}$ as follows:
\[
\begin{array}{ll}
   &  U^{\delta,1}_{k+1}\cap U^{\delta,0}_{k+1}=\\
(\alpha_\cap) &    = \big(H^1_{k+1}\cap H^0_{k+1}\big)\bigcup
\big(H^1_{k+1}\cap\bigcup\{\tilde{C}_x:x\in H^0_{k+1}\}\big)\bigcup
\\
(\beta_\cap) & \bigcup\big(H^1_{k+1}\cap\bigcup\{\tilde{S}^{\la
x,y\ra,n}:n\leq k+1,
\la x,y\ra\in  \\
\ & \in H^0_{k+1}\times (G^1_{k+1}\cup G^0_{k+1}) \}\big)\bigcup \\
 (\gamma_\cap) &  \bigcup  \big(\bigcup\{\tilde{C}_x:x\in H^1_{k+1}\}\cap
H^0_{k+1}\big)\bigcup\\
(\zeta_\cap) & \bigcup
\big(\bigcup\{\tilde{C}_x:x\in H^1_{k+1}\}\cap\bigcup\{\tilde{C}_x:x\in H^0_{k+1}\}\big)\bigcup\\
(\varepsilon_\cap) & \bigcup\big(\bigcup\{\tilde{C}_x:x\in
H^1_{k+1}\}\cap\\
 \ & \cap\bigcup\{\tilde{S}^{\la x,y\ra,n}:n\leq k+1,
\la x,y\ra\in  H^0_{k+1}\times (G^1_{k+1}\cup G^0_{k+1}) \}\big)\bigcup \\
 (\eta_\cap) & \bigcup\big(\bigcup\{\tilde{S}^{\la x,y\ra,n}:n\leq k+1, \la x,y\ra\in \\
 \ & \in H^1_{k+1}\times(G^1_{k+1}\cup G^0_{k+1})\}\cap H^0_{k+1}\big)\bigcup\\
(\theta_\cap) & \bigcup \big(\bigcup\{\tilde{S}^{\la x,y\ra,n}:n\leq
k+1,
\la x,y\ra\in H^1_{k+1}\times(G^1_{k+1}\cup G^0_{k+1})\}\cap\\
\ & \cap\bigcup\{\tilde{C}_x:x\in H^0_{k+1}\}\big)\bigcup \\
(\lambda_\cap) & \bigcup \big(\bigcup\{\tilde{S}^{\la
x,y\ra,n}:n\leq k+1, \la x,y\ra\in
H^1_{k+1}\times(G^1_{k+1}\cup G^0_{k+1})\}\cap  \\
\ & \cap\bigcup\{\tilde{S}^{\la x,y\ra,n}:n\leq k+1, \la x,y\ra\in
H^0_{k+1}\times (G^1_{k+1}\cup G^0_{k+1}) \}\big).
\end{array}
\]
Similarly (but easier) to the proof of $(vi) $ we can check that all
the intersections in the formula displayed above are empty. For
instance, let us consider the intersection in item $(\theta_\cap)$.
Fix  any
$$\la\la x,y\ra,n\ra \in \big(H^1_{k+1}\times(G^1_{k+1}\cup G^0_{k+1})\big)\times (k+2) $$
and $x'\in H^0_{k+1}$. If $x'\not\in H^0_k$, then we have explicitly
subtracted $S^{\la x,y\ra,n}$ from $C_{x'}$ in the definition of
$\tilde{C}_{x'}$, so $\tilde{C}_{x'}\cap\tilde{S}^{\la
x,y\ra,n}=\emptyset$. If $\la\la x,y\ra,n\ra\in N$ (see the proof of
$(vi)$ for the definition thereof), then we have explicitly
subtracted  $C_{x'}$ from $S^{\la x,y\ra,n}$ in the definition of
$\tilde{S}^{\la x,y\ra,n}$, so again
$\tilde{C}_{x'}\cap\tilde{S}^{\la x,y\ra,n}=\emptyset$. Otherwise
$\tilde{S}^{\la x,y\ra,n}\subset U^{\delta,1}_k$ and
$\tilde{C}_{x'}\subset U^{\delta,0}_k$ and therefore we can use
$(iv)$ for $k$ in order to get $\tilde{C}_{x'}\cap\tilde{S}^{\la
x,y\ra,n}=\emptyset$.

 Thus the inductive construction of all the objects
needed for defining $U_\delta$ and $V_\delta$, so that conditions
$(i)-(vii)$ are satisfied, is finished.

As planned, we set  $U_\delta=\bigcup_{k\in\w} U^{\delta,1}_k$,
$V_\delta=\bigcup_{k\in\w} V^{\delta,1}_k$, and note that $\w\setminus U_\delta=\bigcup_{k\in\w} U^{\delta,0}_k$
and $\w\setminus V_\delta=\bigcup_{k\in\w} V^{\delta,0}_k$
by $(iii)$ and $(iv)$. Let $\tau_\alpha$ and $\sigma_\alpha$ be the topologies
generated by $\tau^{-}_\alpha\cup\{U_\delta,\w\setminus U_\delta\}$ and $\sigma^{-}_\alpha\cup\{V_\delta,\w\setminus V_\delta\}$
as a base, respectively. This way we get  $0$-dimensional topologies
and  $(vi)$ implies that $(6)$ is satisfied.
Moreover, $(vii)$ gives that $U_\delta=\bigcup_{k\in\w}H^1_k$, $\w\setminus U_\delta=\bigcup_{k\in\w}H^0_k$,
$V_\delta=\bigcup_{k\in\w}G^1_k$, and $\w\setminus V_\delta=\bigcup_{k\in\w}G^0_k$.
Thus by $(ii)$ for every $x\in U_\delta$ the set $C_x$ is almost contained in $U_\delta$,
and hence also $Y\subset^* U_\delta$ for all $Y\in\Y_\alpha$ converging to $x$.
The same holds for $\w\setminus U_\delta.$
It follows that all $Y\in\Y_\alpha$ remain convergent sequences in $\la \w,\tau_\alpha\ra$.
Analogously, all $Z\in\Z_\alpha$ remain convergent sequences in $\la \w,\sigma_\alpha\ra$,
and hence $(3)$ holds for $\alpha$.
Finally, whenever $n\in\w$, $\ell\in 2^2$, and $\la x,y\ra\in U^{(\ell(0))}_\delta\times V^{(\ell(1))}_\delta$,
then\footnote{For $I\subset\w$ we set $I^{(1)}=I$ and $I^{(0)}=\w\setminus I$.},
then $\{\la s^{\la x,y\ra,n}_k, s^{\la x,y\ra,n}_k\ra:k\in\w\}\subset^* U^{(\ell(0))}_\delta\times V^{(\ell(1))}_\delta$.
Indeed, by $(vii)$ there exists $k\in\w$ such that $x\in U^{\delta,\ell(0)}_k$
and $y\in V^{\delta,\ell(1)}_k$, and therefore $(ii)$
implies
$$ \tilde{S}^{\la x,y\ra,n}\subset^* U^{\delta,\ell(0)}_k\subset U^{\ell(0)}_\delta \mbox{\ \ and\ \ } \tilde{T}^{\la x,y\ra,n}\subset^* V^{\delta,\ell(1)}_k\subset V^{\ell(1)}_\delta.$$
As a result  $E_n$ remains dense in
$\la\w\times\w,\tau_\alpha\otimes\sigma_\alpha\ra$ for all $n$,
because there is an injective sequence in $E_n$ converging to each
$\la x,y\ra$. This also proves that there are no isolated points in
$\la\w\times\w,\tau_\alpha\otimes\sigma_\alpha\ra$. This completes
the verification of $(1)-(7)$ for $\delta=\alpha$.

\textbf{\textit{III}}.  $\alpha=\delta+1$, $\delta$ is odd,   $x_\delta$ is a
limit point of $S_\delta$ in $\la\w,\sigma_\delta\ra$, and $|Z\cap S_\delta|<\w$ for all
$Z\in\Z_\delta$ converging to $x_\delta$.\\
In this case we simply repeat our argument from Case II, with the
roles of $\Z_\delta$ and $\Y_\delta$ interchanged, so that again
$(1)$-$(7)$ are satisfied for  $\delta=\alpha$.

\textbf{\textit{IV}}.  $\alpha=\delta+1$, $\delta$ is even, and either   $x_\delta$ is \emph{not} a
limit point of $S_\delta$ in $\la\w,\tau_\delta\ra$,
or it is and there exists
$Y\in\Y_\delta$ converging to $x_\delta$ in $\la\w,\tau_\delta\ra$
with $|Y\cap S_\delta|=\w$. Then we set $\Y_\alpha=\Y_\delta$, $\Z_\alpha=\Z_\delta$,
$\tau^-_\alpha=\tau_\delta,$ $\sigma^{-}_\alpha=\sigma_\delta$, and repeat the argument from
Case II starting from Claim~\ref{aux1}.

\textbf{\textit{V}}.  $\alpha=\delta+1$, $\delta$ is odd, and either   $x_\delta$ is \emph{not} a
limit point of $S_\delta$ in $\la\w,\sigma_\delta\ra$,
or it is and there exists
$Z\in\Z_\delta$ converging to $x_\delta$ in $\la\w,\sigma_\delta\ra$
with $|Z\cap S_\delta|=\w$. Then we set $\Y_\alpha=\Y_\delta$, $\Z_\alpha=\Z_\delta$,
$\tau^-_\alpha=\tau_\delta,$ $\sigma^{-}_\alpha=\sigma_\delta$, and repeat the argument from
Case III starting from Claim~\ref{aux1}.

All in all, this completes our construction of the objects mentioned
in $(1)$-$(7)$ for all $ \delta<\hot c$, so that these conditions
are satisfied. Conditions $(3)$ and $(2)$ imply that
$X=\la\w,\tau_{\hot c}\ra$ and $Y=\la\w,\sigma_{\hot c}\ra$ are
Frechet-Urysohn, and $X\times Y$ is not $M$-separable by $(6)$.
Moreover, both spaces are zero-dimensional since we started from
zero-dimensional topologies and always enlarged them by adding new
sets along with their complements. This completes our proof of
Theorem~\ref{ma_weaker_pfa}.
\end{proof}

Finally, Theorem~\ref{th:main} is a direct consequence
of Theorems~\ref{(*Q)} and \ref{ma_weaker_pfa}.

\section{Products of $H$-separable spaces in the Laver model} \label{sec5}

We will need to define the following new class of spaces.

\begin{definition}
A topological space $\la X, \tau \ra$  is called \emph{bounded
box-separable} (briefly, \emph{b.b.-separable}) if for every
function $R$ assigning to each countable family $\cU$ of non-empty
open subsets of $X$ a sequence $R(\cU)=\la F_{n} : n \in \omega \ra$
of finite non-empty subsets of $X$ such that $\{n :F_{n} \subset U
\}$ is infinite for every $U \in \cU$, there exists $\mathbb{U}
\subset [\tau \setminus \{ \emptyset \} ]^{\omega}$ of size
$|\mathbb{U}|=\omega_{1}$ and a sequence $\la l_{i} : i \in \omega
\ra \in \omega^{\omega}$ such that for all $U \in \tau \setminus \{
\emptyset \}$ there exists $\cU \in \mathbb{U}$ such that for all
but finite $i \in \omega$ there is $n \in [l_{i},l_{i+1})$ such that
$R(\cU)(n) \subseteq U$.
\end{definition}

The proof of the following statement is  close to that of
\cite[Lemma~2.2]{RepZd18}, the only difference being a more careful
analysis of sets of the form $\{n: R(\cU)(n)\subset U\}$.

\begin{proposition} \label{main_pl}
In the Laver model every countable $H$-separable space is
b.b.-separable \end{proposition}
\begin{proof}
 We work in $V[G_{\omega_{2}}]$,
where $G_{\omega_{2}}$ is $\mathbb{P}_{\omega_{2}}$-generic and
$\mathbb{P}_{\omega_{2}}$ is the iteration of length $\omega_{2}$
with countable supports of the Laver forcing, see \cite{Lav76} for
details. Let us fix an $H$-separable space of the form
$\la\omega,\tau\ra$ and a function $R$ such as in the definition of
b.b-separability. By a standard argument (see, e.g., the proof of
\cite[Lemma 5.10]{BlaShe87}) there exists an $\omega_{1}$-club $C
\subseteq \omega_{2}$ such that for every $\alpha \in C$ the
following conditions hold:
\begin{itemize}
\item[(i)] $\tau \cap V[G_{\alpha}] \in V[G_{\alpha}]$ and
 for every sequence $\la D_{n}:n \in \omega \ra \in V[G_{\alpha}]$ of
 dense subsets of $\la\omega,\tau\ra$ there exists a sequence
  $\la K_{n}:n\in\omega\ra \in V[G_{\alpha}]$ such that
  $K_{n}\in [D_{n}]^{<\omega}$ and for every
  $U \in \tau \setminus \{ \emptyset \}$ the intersection
  $U \cap K_{n}$ is non-empty for all but finitely many $n \in \omega$;
\item[(ii)] $R(\cU)\in V[G_{\alpha}]$ for
any $\cU \in [\tau \setminus \{ \emptyset \}]^{\omega}\cap V[G_{\alpha}]$; and
\item[(iii)] For every $A \in \PP(\omega)\cap V[G_{\alpha}]$
 the interior $Int(A)$ also belongs to $V[G_{\alpha}]$.
\end{itemize}
By \cite[Lemma 11]{Lav76} there is no loss of generality in assuming
that $0 \in C$. Set
$\mathbb{U}:=[\tau\setminus\{\emptyset\}]^{\omega}\cap V$. It
suffices to prove the following auxiliary

\begin{lemma}\label{cl_bb_01}
For every $A \in \tau \setminus \{ \emptyset \}$ there are $\cU \in
\mathbb{U}, J \in [\omega]^{\omega}\cap V[G_{1}]$ such that for
every consecutive $j, j' \in J$ there exists $n \in [j,j')$ with the
property $R(\cU)(n) \subset A$.
\end{lemma}
\begin{proof}
We shall prove an equivalent statement:
\begin{quote}
For every $A \in \tau \setminus \{ \emptyset \}$ there are $\cU \in
\mathbb{U}, J \in [\omega]^{\omega}\cap V[G_{1}]$ such that for
every $m\in\w$  there exists $n \in [m,J(m))$ with the property
$R(\cU)(n) \subset A$, where $J(m)$ is the $m$-th element of $J$
with respect to its increasing enumeration.
\end{quote}

Suppose that  there exists $A \in \tau \setminus \{ \emptyset \}$
for which the lemma is false. Let $\name A$ be a
$\mathbb{P}_{\omega_{2}}$-name for $A$ and
$p\in\mathbb{P}_{\omega_{2}}$ a condition forcing the negation of
the  statement quoted above. Applying \cite[Lemma 14]{Lav76} to the
sequence $\la \name a_{i}:i \in \omega \ra$ such that $\name
a_{i}=\name A$ for all $i\in \omega$, we get a condition $p'\leq p$
such that $p'(0) \leq^{0} p(0)$, and a finite set $\cU_{s} \subset
\PP(\omega)$ for every $s \in p'(0)$ with $p'(0)\la 0\ra \leq s$,
such that for each $n \in \omega$, $s \in p'(0)$ with $p'(0)\la 0\ra
\leq s$, and for all but finitely many immediate successors $t$ of
$s$ in $p'(0)$ we have
\begin{center}
$\extt{p'(0)_{t}}{\hspace{1mm} p'\restriction [1, \omega_{2})} \Vdash \exists U \in \cU_{s} (\name A \cap n = U \cap n)$.
\end{center}
\begin{claim} \label{cl_bb_02}
There exists $p'' \leq p'$ s.t. for every $s \in p''(0), p''(0) \la
0 \ra \leq s$, $n \in \omega$ and every immediate successor $t$ of
$s$ in $p''(0)$ we have $$\extt{p''(0)_{t}}{\hspace{1mm}
p''\restriction [1, \omega_{2})} \Vdash \exists U \in \cU_{s}
(Int(U) \neq \emptyset \land \name A \cap n = U \cap n).$$
\end{claim}
\begin{proof}
Suppose by contradiction that
\begin{itemize}
\item[$(**)$]
  For every $p'' \leq p'$ it is
possible to find $s \in p''(0)$ with $p''(0) \la 0 \ra \leq s$, $n
\in \omega$, infinitely many immediate successors $t$ of $s$ in
$p''(0)$, for which there exists  $r=r(t) \leq
\extt{p''(0)_{t}}{\hspace{1mm} p''\restriction [1, \omega_{2})}$
forcing\\
$\forall U\in\U_s\: (U\cap n=\name{A}\cap n\Rightarrow
Int(U)=\emptyset)$.
\end{itemize}
Let us fix $p''\leq p'$ and let $s$ be as in $(**)$. Note that we
can replace  $n$ with any other bigger number and $(**)$ will be
still satisfied, so we may assume that $U\cap n\neq U'\cap n$ for
any distinct $U,U'\in\U_s$.
 Let $p^{(3)}\in\IP_{\w_2}$ be the
condition obtained by strengthening $p''(0)$ by leaving only those
infinitely many immediate successors of $s$ like in $(**)$, and
$p^{(3)}(\alpha)=p''(\alpha)$ for all $\alpha>0$. Next, removing yet
another finite collection of immediate successors of $s$ in
$p^{(3)}(0)$ and keeping the other coordinates the same, we may
assume that for every immediate successor $t$ of $s$ in $p^{(3)}(0)$
we have
$$ p^{(3)}(0)_{t}^{\,\frown} \hspace{1mm} p''\restriction [1,
\omega_{2}) \forces \exists U\in\U_s\: (U\cap n=\name{A}\cap n),$$
and hence by strengthening $r(t),$ if necessary, we may additionally
assume that for some $U(t)\in\U_s$ with $Int(U(t))=\emptyset$ we
have $r(t) \Vdash \name A \cap n = U(t) \cap n$. By the choice of
$n$ we get that for each $t$ such an
 $U(t)$ is unique. Furthermore, strengthening
 $p^{(3)}$ again in the way described above and
 using the finiteness of $\U_s$, we may assume that there exists $U_s\in\U_s$ with
  $U(t)=U_s$ for all $t$ as above. Summarizing all the modifications
  of $p''$ mentioned above, we get

\begin{itemize}
\item[$(**)'$]
  For every $p'' \leq p'$ it is
possible to find $p^{(3)}\leq p''$, $s \in p^{(3)}(0)$ with
$p^{(3)}(0) \la 0 \ra \leq s$, $n \in \omega$ with $U\cap n\neq
U'\cap n$ for any distinct $U,U'\in\U_s$, $U_s\in\U_s$ with
$Int(U_s)=\emptyset$, and for every immediate successors $t$ of $s$
in $p^{(3)}(0)$ a condition $r(t) \leq
\extt{p^{(3)}(0)_{t}}{\hspace{1mm} p''\restriction [1, \omega_{2})}$
forcing  $U_s\cap n=\name{A}\cap n$.
\end{itemize}
 Let $\la D_{k} : k \in \omega \ra\in V$ be a  sequence of dense
subsets of $\la\omega,\tau\ra$ such that for every $U \in \bigcup \{
\cU_{s} : s \in p'(0), p'(0) \la 0 \ra \leq s \}$ with $Int(U) =
\emptyset$ there are infinitely many $k \in \omega$ such that $D_{k}
= \omega \setminus U$. Let $\la F_{k} : k \in \omega \ra\in V$  be a
witness of the $H$-separability of $X$ for $\la D_{k} : k \in \omega
\ra$. So it is possible to choose $k_{0} \in \omega$ and $p'' \leq
p'$ such that $p'' \Vdash (\forall k \geq k_{0})(\name A \cap F_{k}
\neq \emptyset)$.

Fix $p^{(3)}$, $s,$ $n$, $U_s$, and $p(t)$'s  such as in $(**)'$.
 Let
$k_1 \geq k_{0}$ be  such that $D_{k_1}=\w\setminus U_s$, and thus
$F_{k_1} \cap U_s = \emptyset$. Choose now $n_1
>  n,\max F_{k_1}$
 and an immediate successor $t$ of $s$ in $p^{(3)}(0)$
with
\begin{center}
$\extt{p^{(3)}(0)_{t}}{\hspace{1mm} p''\restriction [1, \omega_{2})}
\Vdash \exists U \in \cU_{s} (\name A \cap n_1 = U \cap n_1)$.
\end{center}
Thus $r(t)$ also forces the above property (being stronger than
$\extt{p^{(3)}(0)_{t}}{\hspace{1mm} p''\restriction [1,
\omega_{2})}$) as well as $U_s\cap n=\name{A}\cap n$. Since all the
differences between elements of $\U_s$ show up below $n$, we
conclude that $r(t)$ forces $U_s\cap n_1=\name{A}\cap n_1$. Since
$$ F_{k_1}\subset D_{k_1}\cap n_1 =(\w\setminus U_s)\cap n_1=n_1\setminus U_s, $$
we conclude that $r(t)$ forces $F_{k_1}\cap\name{A}=\emptyset$. On
the other hand $k_1\geq k_0$, $r(t)\leq p''$, and the latter forces
$\name{A}\cap F_k\neq\emptyset$ for all $k\geq k_0$, a
contradiction.
\end{proof}
We can now assume that every $U \in \cU_{s}$ with $s \in p'(0)$  and
$p'(0) \la 0 \ra \leq s$ has nonempty interior. Let $\cU = \{ Int(U)
: U \in \bigcup \{ \cU_{s} : s \in p''(0), p''(0) \la 0 \ra \leq s
\}$ (or any countable family containing it if this family of
interiors is finite) and $R(\cU)(n)=\la F_{n} : n \in \omega \ra$.
Replace $p'(0)$ with a tree $T$ in the Laver forcing $\mathcal{L}$
such that $T \la 0 \ra = p'(0) \la 0\ra$ and for all $s \in T$ with
$s \geq T\la 0 \ra$ there exists $N_{s} \in \omega$  so that for
every immediate successor $t$ of $s$ in $T$ the following two
properties hold:
\begin{itemize}
\item[(a)] $\forall U \in \cU_{s}$ $\exists n=n(s,U) > |s|\:
(F_{n} \subset Int(U) \cap N_{s})$
\item[(b)] $\extt{T_{t} \hspace{.5mm}}{\hspace{.5mm} p'\restriction [1,\omega_{2})}
\Vdash \exists U \in \cU_{s}\: (\name A \cap N_{s} = U \cap N_{s})$
\end{itemize}
For any $s \in T$ set
\begin{center}
$j_{s} = \max \{ n(s,U) : U \in \cU_{s} \}+1$
\end{center}
Let $G_{1}$
be
$\mathcal{L}$-generic over $V$ with $T \in G_{1}$ and
$\ell\in\zrost$ be the Laver real generated by $G_{1}$. Finally put
\begin{center}
$J = \{j_{ \ell \restriction m} : m \geq |T\la 0 \ra|\}$
\end{center}
and note that $J \in V[G_{1}]$.

\begin{claim} \label{cl_bb_03}
$\extt{T}{\hspace{.5mm} p' \restriction [1,\omega_{2})} \Vdash
\forall m \geq |T \la 0 \ra|\: \exists n \in [m, j_{\name
\ell\restriction m})\: (F_{n} \subset \name A)$
\end{claim}
\begin{proof}
Suppose the statement is false  and pick $r \leq \extt{T}{p'
\restriction [1,\omega_{2})}$ and $m \geq |T \la 0 \ra |$ such that
$$r \Vdash  \forall n \in [m,  j_{\name \ell \restriction m}) (F_{n}
\not\subset \name A).$$
 We can write $r = \extt{R}{\hspace{.5mm}r
\restriction[1,\omega_{2})}$  and without loss of generality assume
$|R \la 0 \ra| \geq m+1$. Setting $\{s\} = R \cap \omega^{m}$ and
$\{t\} = R \cap \omega^{m+1}$ we obtain two elements of $T$ which
contradict the disjunction of $(a)$ and $(b)$ above.
\end{proof}

It follows that $$\extt{T}{\hspace{.5mm} p' \restriction
[1,\omega_{2})}\leq p'\leq p$$ and $\extt{T}{\hspace{.5mm} p'
\restriction [1,\omega_{2})}$ forces for $\name{A}$ the quoted
statement after the formulation of Lemma~\ref{cl_bb_01},
contradicting  our assumption that $p$ forces for $\name{A}$ the
negation of that statement. This contradiction completes the proof
of Lemma~\ref{cl_bb_01} and thus also of Proposition~\ref{main_pl}.
\end{proof}
\end{proof}

\begin{lemma} \label{bbb}
Suppose $\bb > \omega_{1}$, $X$ is a b.b.-separable space and $Y$ is
an $H$-separable space. Then $X \times Y$ is $mH$-separable,  provided
it is separable.
\begin{proof}
Let $\la D_{n} : n \in \omega \ra$  be a decreasing sequence of
countable dense subsets of $X \times Y$. Fix a countable family
$\cU$ of open non-empty subsets of $X$ and a partition
$\{\Omega_{U}:U\in\mathcal U\}$ of $\omega $ into infinite pieces.
For every $U \in \cU$ and $n \in \Omega_{U}$ set
\begin{center}
$D_{n}^{\cU}=\{ y \in Y : \exists x \in U (\la x,y \ra \in D_{n} )
\}.$
\end{center}
Note that every $D_{n}^{\cU}$ defined in  this way is dense in $Y$.
Since $Y$ is $H$-separable there exists a sequence of sets $\la
L^{\cU}_{n} : n \in \omega \ra$ such that $L^{\cU}_{n} \in [
D^{\cU}_{n}]^{< \omega} $ for all $n \in \omega$ and every open
subset of $Y$ intersects all but finite $L^{\cU}_{n}$'s. For every
$U \in \cU$ and $n \in \Omega_{U}$ find a set $K^{\cU}_{n} \in
[U]^{< \omega}$ such that for every $y \in L^{\cU}_{n}$ there exists
$x \in K^{\cU}_{n}$ such that $\la x,y \ra \in D_{n}$, and set
$R(\cU)=\la K^{\cU}_{n} : n \in \omega \ra$. Note that $K^{\cU}_{n}
\subset U$ for all $n \in \Omega_{U}$, so $R$ is as in the
definition of b.b.-separability. Therefore there exist a family
$\mathbb{U}$ of countable collections of open non-empty subsets of
$X$ of cardinality $\omega_{1}$ and a sequence $\la l_{i} : i \in
\omega \ra \in \omega^{\omega}$ that witness the b.b.-separability.
By our hypothesis $|\mathbb U| < \bb$, so we can select a sequence
$\la F_{n} : n \in \omega \ra$ such that $F_{n} \in [ D_{n}]^{<
\omega} $ and $K^{\cU}_{n} \times L^{\cU}_{n} \subset F_{n}$ for all
$\mathcal U\in\mathbb U$ and all but finitely many $n \in \omega$.
We claim that the sequence
$$\la F'_i:=\bigcup\limits_{n
\in [l_{i},l_{i+1})} F_{n}\: :\: i\in\w\ra$$
 witnesses the
$mH$-separability of $X \times Y$. First of all, note that
$F'_i\subset D_{l_i}\subset D_i$ by the monotonicity of the sequence
$\la D_n:n\in\w\ra$. Now fix an open non-empty subset $U \times V
\subseteq X \times Y$ and find $\cU \in \mathbb{U}$ and $i_0\in\w$
such that for all $i \geq i_{0}$ there exists $n \in
[l_{i},l_{i+1})$ with $R(\cU)(n)=K^{\cU}_{n} \subset U$. Suppose
that $i \geq i_{0}$, and both of the conditions $K^{\cU}_{n} \times
L^{\cU}_{n} \subset F_{n}$ and  $L^{\cU}_{n} \cap V \neq \emptyset$
hold true for all $n \geq l_{i}$. Given $n\in [l_i,l_{i+1})$ with
$R(\cU)(n)=K^{\cU}_{n} \subset U$, it follows from the above that
 $F_{n} \cap (U \times V) \neq \emptyset$, which
combined with $F_n\subset F'_i$
  proves the $mH$-separability of $X \times
Y$.
\end{proof}
\end{lemma}
Proposition~\ref{main_pl}, Lemma~\ref{bbb}, and the fact that
$\mathfrak b>\w_1$ is true in the Laver model all together imply the following
\begin{theorem}\label{main_tl}
In the Laver model, the product of two $H$-separable spaces is
$mH$-separable provided that it is separable. In particular, in this
model the product of two countable $H$-separable spaces is
$mH$-separable.
\end{theorem}


\end{document}